\pdfoutput=1 
\pdfoutput=1 
\pdfoutput=1 
\pdfoutput=1 
\pdfoutput=1 
\documentclass[10pt, a4paper, oneside, reqno]{amsart}

\makeatletter
\def\@settitle{\begin{center}%
		\baselineskip14\p@\relax
		\normalfont\LARGE\scshape\bfseries
		\@title
	\end{center}%
}
\makeatother

\makeatletter

\def\subsection{\@startsection{subsection}{2}%
	\z@{.5\linespacing\@plus.7\linespacing}{.5\linespacing}%
	{\normalfont\large\bfseries}}

\def\subsubsection{\@startsection{subsubsection}{3}%
	\z@{.5\linespacing\@plus.7\linespacing}{.5\linespacing}%
	{\normalfont\itshape}}

\usepackage[usenames, dvipsnames]{color}
\definecolor{darkblue}{rgb}{0.0, 0.0, 0.45}

\usepackage[colorlinks	= true,
raiselinks	= true,
linkcolor	= darkblue, 
citecolor	= Mahogany,
urlcolor	= ForestGreen,
pdfauthor	= {Peyman Mohajerin Esfahani},
pdftitle	= {},
pdfkeywords	= {},
pdfsubject	= {},
plainpages	= false]{hyperref}

\usepackage{dsfont,amssymb,amsmath,subcaption, graphicx,enumitem} 
\usepackage{amsfonts,dsfont,mathtools, mathrsfs,amsthm} 
\usepackage[amssymb, thickqspace]{SIunits}
\usepackage{algorithm,algorithmicx,fancyhdr}
\usepackage{multirow}  

\allowdisplaybreaks
\date{\today}

\theoremstyle{theorem}
\newtheorem{Thm}{Theorem}[section]

\newtheorem{Cor}[Thm]{Corollary}

\newtheorem{Def}[Thm]{Definition}

\newtheorem{Rem}[Thm]{Remark}

\theoremstyle{remark}

\definecolor{green}{rgb}{0,.6,0}

\newcommand{\R}{\mathbb{R}}

\newcommand{\Let}{\coloneqq}

\newcommand{\RNum}[1]{\uppercase\expandafter{\romannumeral #1\relax}}

\newcommand{\basis}{F_{\mathrm b}}

\graphicspath{{./fig/}}

\title[Anomaly Detection with High-fidelity Simulators]
{Dynamic Anomaly Detection with High-fidelity  
\\Simulators: A Convex Optimization Approach}

\author{Kaikai~Pan,
	Peter~Palensky,
	and~Peyman~Mohajerin~Esfahani}%
	\thanks{The authors are with the Delft University of Technology, The Netherlands (email: \tt{\{K.Pan, P.Palensky,  P.MohajerinEsfahani\}}@tudelft.nl).}

\begin{document} 
\maketitle

\begin{abstract}

The main objective of this article is to develop scalable dynamic anomaly detectors 
when high-fidelity simulators of power systems are at our disposal. On the one hand, mathematical models 
of these high-fidelity simulators are typically 
``intractable'' 
to apply existing model-based approaches. 
On the other hand, pure data-driven methods developed primarily in the machine learning literature neglect our knowledge about the underlying dynamics of 
the systems. 
In this study, 
we combine tools from these two mainstream approaches to develop a 
diagnosis filter 
that utilizes the knowledge of both the dynamical system as well as the simulation data of the high-fidelity simulators. 
The proposed diagnosis filter aims to achieve two desired features: (i) performance robustness with respect to 
model mismatch; (ii) high scalability. To this end, we propose a tractable (convex) optimization-based reformulation in which decisions are the filter parameters, the model-based information introduces feasible sets, and the data from the simulator forms the objective function to-be-minimized regarding the effect of 
model mismatch on the filter performance. To validate the theoretical results and its effectiveness, 
we implement the developed diagnosis filter in DIgSILENT PowerFactory to detect false data injection attacks on the Automatic Generation Control measurements in the three-area IEEE 39-bus system. 	

\end{abstract}



\section{Introduction} 
\label{sec:intro}

The principle of anomaly detection in power system cyber security is to generate a diagnostic signal (e.g., residual) which keeps sensitive to malicious intrusions and simultaneously robust against other unknowns, given the available data from system outputs. The detection methods can be mainly classified into two categories: (i) the approaches that exploit an explicit mathematical model of the system dynamics (referred to {\em model-based} methods in this article); (ii) the {\em data-driven} approaches that try to automatically learn the system characteristics from the output data \cite{Ozay2016, Wei2018}. Our work in \cite{Pan2020} has developed a scalable diagnosis tool 
to detect the class of multivariate false data injection (FDI) attacks that may remain stealthy in view of a static detector, by capturing the dynamics signatures of such a disruptive intrusion. This method is, indeed, model-based that the dynamics of the system trajectories under multivariate FDI attacks are described via an explicit mathematical model representation (i.e., linear differential-algebraic equations (DAEs)). The numerical results in \cite{Pan2020} have proven its effectiveness in the given 
linear mathematical model. Now here comes another question: 

\begin{flushleft}
	\centering
	{\em Can the power of scalable model-based diagnosis tools be still utilized in real-world applications such as electric power systems for which there are reliable datasets from high-fidelity, but complex, simulators?} 
\end{flushleft}%

Let us clarify the terminologies 
adopted throughout this article. 
An abstract model refers to an explicit, but perhaps reduced-order, mathematical description of power system dynamics, e.g., 
a linear DAE. 
The datasets are the simulation results from a high-fidelity simulator like DIgSILENT PowerFactory (PF). A simulator is said to be high-fidelity 
when it is the closest to the reality and may consist of several complex nonlinear DAEs as parts. However, unfortunately, one may not have access to 
the mathematical description of such simulator. We aim to address the question above by designing a scalable diagnosis tool 
that can be applied to high-fidelity simulators. To do that, 
we need 
two important pieces of information: the knowledge of 
an abstract model, and 
the 
system trajectories 
provided by a high-fidelity simulator. 
The model-based knowledge 
is utilized for a scalable design such that the design parameters (e.g., the order of the diagnosis tool) should be adjustable depending on the size or degree of the studied system. 
The source of challenge to answer the question comes from the following aspect: 
whatever 
abstract linear model we can pick, {\em 
	model mismatch} 
is always 
reflected through the difference of the output of the abstract model and the one from the high-fidelity simulator. It can be expected that this unknown of 
model mismatch 
potentially affects the diagnostic performance. 
With that in mind, we propose a diagnosis tool that is 
robust with respect to 
model mismatch, by exploiting the information that is revealed to us through 
the simulation data. 
We will provide further details toward this objective in the following Section~\ref{sec:outline_solution}.

\paragraph*{\bf Literature on model-based and data-driven anomaly detection} 
Let us briefly overview the advantages and limitations of the pure model-based and data-driven approaches. 
The model-based methods require 
detailed information of the system dynamics. 
A major subclass of these schemes is the observer-based residual generator 
that historically emerges from a control-theoretic 
perspective and has been extended to 
linear DAEs by \cite{Nyberg2006}. 
To our best of knowledge, the 
study \cite{Taha2016} is the first attempt to apply observer-based detectors to power system cyber security and monitoring problem. Recently, a variant of observer-based method is employed in \cite{Ameli2018a} so as to deal with unknown exogenous inputs in the linear Automatic Generation Control (AGC) system. Parameter estimation model-based approaches have also been extensively investigated. For instance, the extended Kalman filter algorithm is used to perform such an estimation for anomaly detection \cite{Khalaf2019}. In \cite{Qi2018}, a comparison study is carried out for various Kalman filters and observers in power system dynamic state estimation with model uncertainty and malicious cyber attacks. The residual generators above usually have the same degree as the system dynamics, which can be problematic in the online implementation particularly for large-scale power systems \cite{Ding2008}. Our diagnosis filter in \cite{Pan2020} provides a good alternative to detect multivariate FDI attacks in a real-time operation. Still, the challenge remains as the power system models are mostly nonlinear, complex and high-dimensional. 
The work in \cite{Esfahani2016} proposed an optimization-based filter for detecting a single anomaly in the control system where the nonlinearity can be fully described in DAEs. However, as noted earlier, 
having a 
detailed mathematical description of the model especially in the high-fidelity simulator is usually infeasible. 

Another major technique for anomaly detection comes from data-driven approaches which do not require an explicit mathematical model of system dynamics. Developments such as sensing technology, Internet-of-Things and Artificial Intelligence have contributed to a more data-driven power system \cite{Palensky2017a}. Anomaly detection is mainly considered as a classification problem and there are supervised or unsupervised learning approaches for that purpose. Among all the supervised classifications, deep neural networks (DNN) \cite{Yu2018, Ayad2018}, bayesian networks \cite{Wadhawan2018} and Kernel machines \cite{Ozay2016} are the popular methods. For unsupervised classifications for detecting cyber attacks in smart grids, one can find principle component analysis (PCA) and its extension 
\cite{Hao2015}, autoencoders \cite{Ahmed2019}, etc. In addition to supervised or unsupervised approaches, recent work in \cite{Kurt2019} has proposed a reinforcement learning based algorithm for online attack detection in smart grids without a prior knowledge of system models or attack types. Overall, data-driven methods are suitable for real implementations in complex and large-scale power systems. However, their performance highly depends on the quantity and quality of the accessible data, and thus can be intractable in many cases \cite{Tidriri2016}. Besides, the required pre-processing stage (e.g., data training) may have a high computational cost.

\paragraph*{\bf Contributions and outline} 
This article aims to develop a scalable and robust diagnosis filter that can be applied to a high-fidelity simulator like PF. To achieve that, we propose a tractable optimization-based reformulation where the abstract model-based information introduces feasible sets, and the simulation data forms the objective function to minimize the effect of 
model mismatch on the filter 
residual. 
In this way, the diagnosis filter can be ``trained'' in the normal operations (without attacks) to have performance robustness with respect to 
model mismatch. Then the diagnosis filter can be ``tested'' in PF to detect attacks. Our main contributions are: %

\begin{itemize}
	\item[(i)] Firstly, we develop a data-assisted model-based approach that utilizes both the model-based knowledge and also the simulation data from the simulator, for a scalable and robust design (Definition~\ref{def:robust_detect_mis} and the program~\eqref{opt:robust-poly-mis}). Instead of using any existing (un)supervised learning algorithms, we propose our own optimization-based characterization to ``train'' the filter 
	under 
	multiple 
	mismatch signatures 
	obtained through 
	the simulation data (Remark~\ref{rem:robust_filter_univ}). 
	As far as we know, this is the first study that builds on such a perspective. In the optimization-based reformulation, the objective is to minimize the effect of model mismatch on the filter residual. 
	We show that, the resulted optimization programs are convex and hence tractable, indicating that the proposed filter is not computational expensive compared to many machine-learning type methods. 
	
	\item[(ii)] We investigate optimization-based characterizations of the developed diagnosis filter in both scenarios of univariate and multivariate attacks. 
	A square of $\mathcal{L}_{2}$-inner product with corresponding norm is proposed to quantify the effect of 
	model mismatch. Then the $\mathcal{L}_{2}$-norm of the residual part introduced by 
	model mismatch is reformulated as a quadratic function. We prove that, the characterization of the filter under a univariate attack becomes a family of convex quadratic programs (QPs), and the 
	obtained filter can even have the capability of tracking the attack magnitude through its non-zero steady-state residual (Theorem~\ref{thm:qp_univ}). We also extend to 
	the scenario of multivariate attack, and it appears that a standard QP can be derived. Besides, we provide conditions under which the filter can detect all the plausible multivariate attacks in an admissible set (Corollary~\ref{cor:qua_opt_transient}). 
	
\end{itemize}
The process of diagnosis filter construction and validation is concluded in Algorithm~\ref{alg:filterconstr_algorithm}. 
The effectiveness of the proposed 
approach is validated on the three-area IEEE 39-bus system. Numerical results from the case study illustrate that the developed data-assisted model-based filter 
for the PF simulator can successfully generate alerts in the presence of FDI attacks, while a pure model-based detector may fail. 

Section~\ref{sec:outline_solution} shows the outline of our proposed solution. Both the motivating case study and the mathematical framework of our solution are detailed presented. Section~\ref{sec:novel_detect} proposes the tractable optimization-based characterization of the developed diagnosis filter which has high scalability and performance robustness to 
model mismatch. Numerical results of the developed filter comparing with other pure model-based ones are reported in Section~\ref{sec:results}, and conclusion is in Section~\ref{sec:conclusion}.  


\section{Outline of the Proposed Solution} \label{sec:outline_solution}

In this section, we starts from a case study motivating the setting of our study. 
Then the mathematical framework of our proposed solution to design a diagnosis filter that can be applied to the high-fidelity simulator is presented. %

\begin{figure}[t!p]
	\centering
	\includegraphics[scale=0.0387]{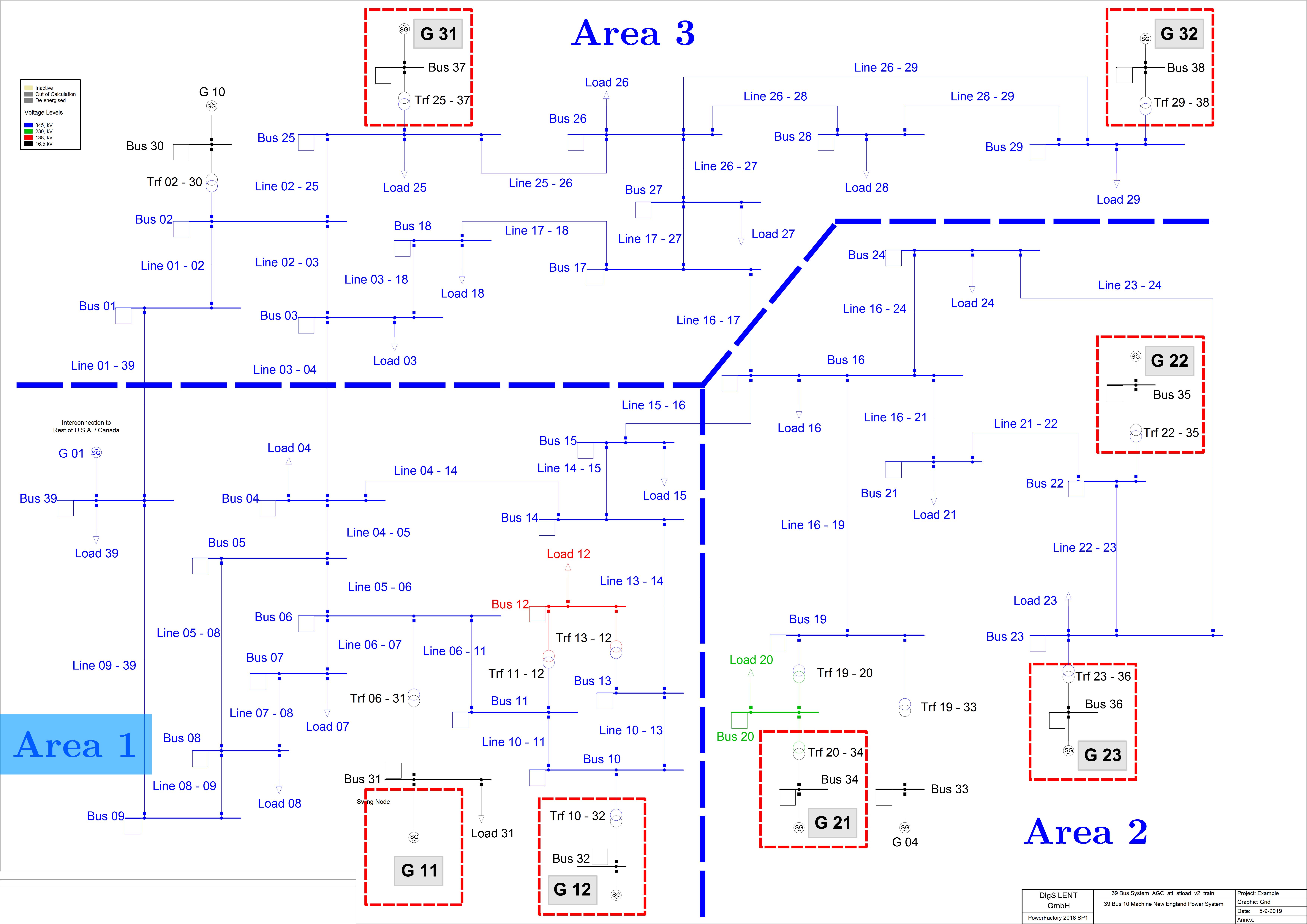}
	\caption{Three-area power system with AGC functions 
		in high-fidelity simulator PF.}
	\label{fig:39bus_pf}
\end{figure}

\subsection{Motivating case study} \label{subsec:motivation}%

Let us consider a three-area power system equipped with AGC functions in Figure~\ref{fig:39bus_pf}. The three-area IEEE 39-bus system is modeled in the PF simulator. In PF, the dynamic generator model would consist of a synchronous machine, along with a voltage regulator for the exciter, and also the turbine-governor unit. We also implement AGC in each area for secondary frequency control. The AGC is a typical automatic control loop that regulates the system frequency and the power exchanges between areas by controlling the power settings of the generators participating in AGC to follow load changes. 
Looking into the behavior of the electromechanical dynamics in PF, we can describe the formal mathematical model of the studied system by a set of DAE equations as
%
%
\begin{subequations}\label{eq:nonDAEmodel}
	\begin{align}
	& \dot{x}(t) = \mathsf{f}(x(t), \, a(t), \, f(t)) \, , \label{eqsub:dotx_f}\\
	& 0 = \mathsf{g}(x(t) \, , a(t) \, , d(t)), \label{eqsub:g} 
	\end{align}
\end{subequations}
where $x \in \mathbb{R}^{n_{x}}$ is 
the vector of augmented state variables including the ones of synchronous machines (e.g., damper-windings, mechanical equations of motion, exciter,
voltage regulator, turbine-governor unit) and AGC controllers \cite{Chakrabortty2011}. The vector $a \in \mathbb{R}^{n_{a}}$ represents the 
algebraic variables, and the vector $d \in \mathbb{R}^{n_{d}}$ denotes the 
load disturbances. 
We note that \eqref{eqsub:dotx_f} consists of both synchronous generator dynamics and AGC controller dynamics. 
The AGC controller collects the information of grid frequency and power exchange on the tie-line (e.g., L1-2 in Figure~\ref{fig:39bus_pf}) that connects areas to form the area control error (ACE) to be minimized. Then the power settings are computed for the generation allocation logic of the generators participating in AGC. In practice, the data to form the ACE signal are usually transmitted through unprotected channels \cite{Ten2008,Pan2017a}. 
Thus, we consider an anomaly scenario where FDI attacks are corrupting the ACE; this anomaly of FDI attack is characterized by the vector $f \in \mathbb{R}^{n_{f}}$ in \eqref{eqsub:dotx_f}.

In the nonlinear DAE formulation of \eqref{eq:nonDAEmodel}, the function $\mathsf{f}: \mathbb{R}^{n_{x} + n_{a} + n_{f}} \rightarrow \mathbb{R}^{n_{x}}$ in \eqref{eqsub:dotx_f} captures the dynamics of synchronous generators and 
AGC controllers. 
The function $\mathsf{g}: \mathbb{R}^{n_{x} + n_{a} + n_{d}} \rightarrow \mathbb{R}^{n_{a}}$ in \eqref{eqsub:g} models the electrical network. These nonlinear functions involve saturation, sinusoidal terms and possibly other nonlinearity in the power system. Since \eqref{eq:nonDAEmodel} is for a model description in the high-fidelity PF simulator, one may not have access to the detailed DAE of \eqref{eq:nonDAEmodel}. In what follows, we show that, to design a diagnosis filter for the high-fidelity PF simulator, we do not need an explicit model of \eqref{eq:nonDAEmodel}, however, it is this simulator that gives us the trajectory of the system and the output from the simulations to be fitted into the diagnosis filter for anomaly detection.

Firstly, if the model in \eqref{eq:nonDAEmodel} is known, a common practice for the analysis is to linearize \eqref{eq:nonDAEmodel} under an assumption that the system works closely around the nominal operating point; one can find such treatments in many literature sources like \cite{Chakrabortty2011,Zhang2014}. Next, though \eqref{eq:nonDAEmodel} in PF is unknown to the diagnosis filter design, we can pick an 
abstract linear model whose parameters are obtained from a simplified version of the system. 
If the mathematical description of the complex DAE model \eqref{eq:nonDAEmodel} would be available, then a reasonable choice of this abstract linear model could be the linearized DAE around the operation point. However, we emphasize that this choice does not need to be a linearized version of the DAE, particularly in view of the robustification technique introduced in the next step. Such an abstract linear model can be described by

\begin{equation}\label{eq:linearSP}
\begin{aligned}
& \dot{\tilde{x}}(t) = {A}_{c} \tilde{x}(t) + {B}_{c,d}{d}(t) + {B}_{c,f}{f}(t), \\
& {y}(t) = {C}\tilde{x}(t) + {D}_{f}{f}(t),\\
\end{aligned}
\end{equation}
where $\tilde{x} \in \mathbb{R}^{n_{\tilde{x}}}$ is the reduced-order state vector of the abstract model 
including the dynamics of the multi-area system and also AGC. The vector $y \in \mathbb{R}^{n_{{y}}}$ denotes the system output of \eqref{eq:linearSP}. The model \eqref{eq:linearSP} is picked because, as noted by \cite{kundur1994power}, in AGC we pay more attention on collective performance of all generators, and hence each area is represented by a simplified model comprised of equivalent turbine-governor units and generators. Besides, from the timescales of interest, the frequency response of AGC can be decoupled from the loop of voltage regulator in exciter. Of course, the model \eqref{eq:linearSP} is not meant to be ``low-fidelity'', instead, it can be sufficiently accurate for analytical analysis according to \cite{Pan2020,Ameli2018a}.

We refer to the 
Appendix A for further details of system modeling in PF and the implementation of AGC 
under FDI attacks, and also the parameters in the abstract 
model \eqref{eq:linearSP}. As mentioned in Section~\ref{sec:intro}, the challenge is, 
regardless of the choice of the abstract model \eqref{eq:linearSP}, there would always exist 
{\em model mismatch} between \eqref{eq:nonDAEmodel} and \eqref{eq:linearSP}. 
For instance, we notice that the assumption of operating point 
can be easily violated due to switches and natural disturbances. 

\subsection{Mathematical framework of the proposed solution} \label{subsec:math_frame}

\begin{figure}[t!p]
	\centering
	\includegraphics[scale=1.05]{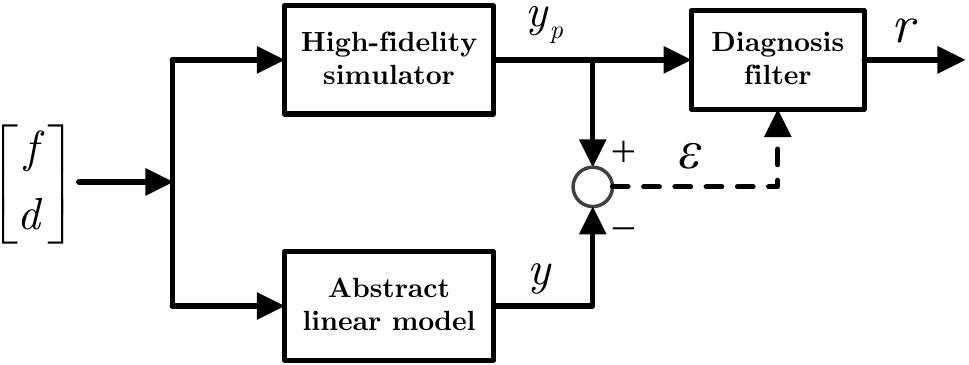}
	\caption{Configuration of the proposed solution. } 
	\label{fig:outline}
\end{figure}
%

To alleviate the impact of model mismatch, 
we propose the solution outlined in Figure~\ref{fig:outline}. We denote the output from the PF simulator as ${y}_{p}$. Our proposed solution builds on a new 
perspective that the diagnosis filter utilizes not only the abstract model-based information for a scalable design but also the simulation data to ``train'' the filter to achieve performance robustness with respect to model mismatch. 
Note that model mismatch is reflected through the difference of the output ${y}_{p}$ from PF and ${y}$ from the abstract model \eqref{eq:linearSP}; such output mismatch signatures are
characterized by $\varepsilon$ in Figure~\ref{fig:outline}. 

Let us start with 
the first approach that builds on 
model information. 
In a realistic framework, the output data are applied to the diagnosis filter in discrete-time samples. Thus firstly, we would like to express the discrete-time version of \eqref{eq:linearSP} as a more compact linear DAE (here it refers to difference algebraic equation) formulation,
\begin{equation}\label{eq:linearDAE}
{H}(q)\bar{x}[k] + {L}(q){y}[k] + {F}(q){f}[k] = 0, \ \forall k \in \mathbb{N},
\end{equation}
where $q$ is a 
time-shift operator such that $q\tilde{x}[k] \rightarrow \tilde{x}[k+1]$. The augmented vector $\bar{x}:= [\tilde{x}^\top \ {d}^\top]^{\top}$ 
consists of both the state variables and the load disturbances. We denote $n_{r}$ as the number of rows in \eqref{eq:linearDAE}. Then ${H}, \ {L}, \ {F}$ are polynomial matrices in terms of the operator $q$ with $n_{r}$ rows and $n_{x}, n_{y}, n_{f}$ columns, respectively, by defining,
\begin{equation}\label{eq:dae_def}
{H}(q) := \left[\begin{matrix} -q{I} + {A} & {B}_{d}\\ {C} & {0}\\ \end{matrix}\right], \, {L}(q) := \left[\begin{matrix} {0} \\ -{I} \\ \end{matrix}\right], \, {F}(q) := \left[\begin{matrix} {B}_{f}\\ {D}_{f} \\ \end{matrix}\right]. \nonumber 
\end{equation}
where $A$, $B_{d}$ and $B_{f}$ are 
from a zero-order hold (ZOH) discretization of \eqref{eq:linearSP} for a sampling time $T_{s}$, according to \cite[Chapter 5-5, page 314]{Ogata1995}.

Next, for the 
filter design in Figure~\ref{fig:outline}, the output ${y}_{p}$ from the 
PF simulator is available as the input to the diagnosis filter. In this article, we propose a 
diagnosis filter as a type of residual generator which has a linear transfer operation, 
%
%
\begin{align}
& r[k] := {R}_{\varepsilon}(q){y}_{p}[k], \label{eq:filter_rg}\\
& \varepsilon[k] = {y}_{p}[k] - {y}[k], \label{eq:output_mismatch}
\end{align}
where $r$ is the residual signal for diagnosis in Figure~\ref{fig:outline}, ${R}_{\varepsilon}(q)$ with a predefined degree is the design variable of the diagnosis filter, which will depend on the 
information of the abstract model in the preceding subsection and also the simulation data 
through the output mismatch signature $\varepsilon$ in \eqref{eq:output_mismatch}.

\section{Optimization-based Characterization of the Diagnosis Filter} 
\label{sec:novel_detect}

\subsection{Robust diagnosis filter}\label{subsec:robust_detect}%

Considering the model-based information in the compact formulation of \eqref{eq:linearDAE}, the residual generator can be represented through the polynomial matrix equations. Thus for 
${R}_{\varepsilon}(q)$, we introduce ${R}_{\varepsilon}(q) := {a}(q)^{-1}{N}(q){L}(q)$. Then ${N}(q)$ with the dimension of $n_{r}$ and a predefined degree $d_{N}$ becomes the filter design variable, if the scalar polynomial ${a}(q)$ with sufficient order to make ${R}_{\varepsilon}(q)$ physically realizable is determined. Note that $d_{N}$ is an adjustable variable to be much less than the order of system dynamics in \eqref{eq:linearSP}. From \eqref{eq:linearDAE} to 
\eqref{eq:output_mismatch}, we can have %
\begin{align}\label{eq:residual_gen}
r[k] &= a(q)^{-1}{N}(q){L}(q){y}_{p}[k] \nonumber\\
& = a(q)^{-1}{N}(q){L}(q)({y}[k]+\varepsilon[k])   \nonumber\\
&= -\underbrace{a(q)^{-1}{N}(q){H}(q)\bar{x}[k]}_{\text{(\RNum{1}})} - \underbrace{a(q)^{-1}{N}(q){F}(q){f}[k]}_{\text{(\RNum{2}})} \nonumber\\ 
& \ \ \ +\underbrace{a(q)^{-1}{N}(q){L}(q)\varepsilon[k]}_{\text{(\RNum{3}})} \, ,
\end{align}
where term (II) is the desired contribution from the anomaly ${f}$. Ideally, we would like to let the residual keep robust against the unknowns of term (I) and (III). For that purpose, first we need to quantify the effect of output mismatch $\varepsilon$ on the filter residual. For all $k \in \mathbb{N}$, let us define  
\begin{align} \label{eq:residual_e}
r_{\varepsilon}[k] := a(q)^{-1}{N}(q){L}(q)\varepsilon[k] \, . %
\end{align}
Next, let us further denote the space of a discrete-time signal taking values in $\mathbb{R}^{n}$ over the horizon of $T$ (i.e., $k \in \{1, \, \cdots, \, T \}$) 
by $\mathcal{{W}}_{T}^{n}$. 
We equip this space with an inner product and a corresponding norm as
\begin{equation}\label{eq:r_e_norm}
\| {v} \|_{\mathcal{L}_{2}}^{2} := \big\langle {v}, \ {v} \big\rangle, \quad \big\langle {v}, \ {w} \big\rangle: = \sum\limits_{k=1}^{T}{v}^{\top}[k]{w}[k], %
\end{equation}
where ${v}, \, {w}$ are some elements in the space $\mathcal{{W}}_{T}^{n}$. 

The main objective of applying a residual generator in PF is to make the residual $r$ as sensitive to anomaly $f$ as possible and simultaneously as robust as possible against other unknowns in term (I) and (III). To achieve that, we introduce a scalable and robust diagnosis filter characterized by a class of residual generator which has the following features.   
\begin{Def}[Robust diagnosis filter]\label{def:robust_detect_mis}
	Consider the residual generator represented via a polynomial vector~${N}(q)$ for a given $a(q)$. This residual generator is robust with respect to 
	output mismatch and can detect all the plausible disruptive attacks, %
	if ${N}(q)$ is the optimal solution from 
	\begin{align}\label{opt:robust-poly-mis}
	\min\limits_{{N}(q)}\quad & \| r_{\varepsilon} \|_{\mathcal{L}_{2}}^{2} 
	\nonumber \\
	\mbox{s.t.}\quad & N(q)H(q) = 0 \, , \\
	&  N(q)F(q)f \neq 0, \ \,  \forall \, f \in \mathcal{F}, \nonumber
	\end{align}
	where $\mathcal{F}$ in term (II) is an admissible anomaly set.
	We emphasize that $\mathcal{F}$ can be adjusted according to different anomaly scenarios where the convexity of the set is particularly desired from a computational perspective. For instance, in the scenario of disruptive univariate attack ($n_f = 1$), one can let $\mathcal{F}$ be a set of $\{f \in \mathbb{R} : \ f_{min} \leq f \leq f_{max} \}$ where $f_{min}, f_{max} \in \mathbb{R}$ are non-zero variables decided by the attack targets on attack impact and undetectability. Similarly, for multivariate attacks ($n_f > 1$), we can introduce a set  
	\begin{align*}
	\mathcal{F} = \big\{f \in \mathbb{R}^{n_{f}} : f = F_{\mathrm b}^{\top} \alpha, \ 
	\alpha \in \mathcal{A} \big\} \, ,
	\end{align*}
	where $F_{b} \Let [f_{1}, \, f_{2}, \, \dots, \, f_{d}]$ represents a finite basis for the set of disruptive attacks and $\alpha \Let [\alpha_{1}, \, \alpha_{2}, \, \cdots, \, \alpha_{d}]^\top \in \mathbb{R}^d$ contains the coefficients. $\mathcal{A} : = \{ \alpha \in \mathbb{R}^{d} \ | \ A\alpha \geq b  \}$ with 
	$A \in \mathbb{R}^{n_{b}\times d}$ and $b \in \mathbb{R}^{n_{b}}$ 
	is a polytopic set to reflect attack targets on attack impact and undetectability. 
\end{Def}

\subsection{Tractable optimization-based characterization under univariate attack} \label{subsec:detect_univ}

In light of \eqref{opt:robust-poly-mis} in Definition~\ref{def:robust_detect_mis} for the robust filter design, let us first consider the univariate attack scenario. Note that when there is no attack, the system output $y_{p}$ from the PF simulator and $y$ from the abstract model \eqref{eq:linearSP} only depend on the input of load disturbances $d$. Thus for one instance of load disturbances, $d_{i}$, one can have a specific output mismatch signature $\varepsilon_{i}$ according to \eqref{eq:output_mismatch}. For each $\varepsilon_{i} \in \mathcal{W}_{T}^{n_{y}}$, a {\em mismatch signature matrix} $\mathcal{E}_{i} \in \R^{n_{y} \times T}$ can be introduced, %
\begin{equation}\label{eq:mismatch_sig}
\mathcal{E}_{i} := \Big[\varepsilon_{i}[1], \ \varepsilon_{i}[2], \ \cdots, \ \varepsilon_{i}[T] \Big] \, .
\end{equation}
%
%

Recall that the operator $q$ acts as a time-shift operator: $q \, \varepsilon_{i}[k] \rightarrow \varepsilon_{i}[k+1]$. 
This operator is linear, and it can be translated as a matrix left-shift operator for matrix $\mathcal{E}_{i}$: $q \, \mathcal{E}_{i} = \mathcal{E}_{i}D$ where $D$ is a square matrix of order $T$. Following the definition of the residual $r_{\varepsilon}$ in \eqref{eq:residual_e}, we have 
\begin{equation}\label{eq:NE}
a(q)r_{\varepsilon} = N(q)L(q)\mathcal{E}_{i}  
= \bar{N}\bar{L} \left[\begin{matrix} I \\ qI \\ \vdots \\ q^{d_{N}}I \end{matrix}\right] \mathcal{E}_{i} = \bar{N}\bar{L}{D_{i}} \, ,
\end{equation}
where the matrices are defined as $\bar{N} := [N_{0}, \, N_{1}, \, \cdots, \, N_{d_{N}} ]$, $\bar{L} := diag [L, \, L, \, \cdots, \, L]$ and $L = L(q)$, and 
$D_{i} : = [\mathcal{E}_{i}^{T}, \ (\mathcal{E}_{i}D)^{T}, \ \cdots, \ (\mathcal{E}_{i}D^{d_{N}})^{T}]^{T}$. Given a particular disturbance pattern $d_i$, then the $\mathcal{L}_2$-norm of the residual as defined in \eqref{eq:r_e_norm} can be reformulated as a quadratic function,  
\begin{equation}\label{eq:l2_nonlinear}
\| r_{\varepsilon_{i}} \|_{\mathcal{L}_{2}}^{2} = \bar{N} Q_{i} \bar{N}^{\top}, \ \ \ Q_{i} = (\bar{L}D_{i}) G (\bar{L}D_{i})^{\top},
\end{equation}
where $G$ is a positive semi-definite matrix with a dimension of $T$ such that $G(i,j) = \big\langle a(q)^{-1} b_i, \ a(q)^{-1}b_j \big\rangle $ in which $b_{i}, b_{j} \in \mathcal{W}_{T}^{1}$ are the discrete-time unit impulses. It can be observed from \eqref{eq:l2_nonlinear} that the matrix $Q_{i}$ is also positive semi-definite since $Q_{i}$ is symmetric and for all non-zero row vector $\bar{N}$, we have $\bar{N} Q_{i} \bar{N}^{\top} = \| r_{\varepsilon_{i}} \|_{\mathcal{L}_{2}}^{2} \geq 0$.

\begin{Rem}[Training with multiple output mismatch signatures]
	\label{rem:robust_filter_univ}
	In order to robustify the diagnosis filter, it can be {\em ``trained''} 
	by utilizing the information of 
	multiple instances of load disturbances, i.e., $\big\{d_{i}\big\}_{i=1}^{m}$, in the normal system operations (without attacks). For each disturbance signature $d_{i}$, the output mismatch signature $\varepsilon_{i}$ and also the matrices $\mathcal{E}_{i}$, $D_{i}$, $Q_{i}$ can be computed from \eqref{eq:mismatch_sig} to \eqref{eq:l2_nonlinear}. Next, according to \eqref{opt:robust-poly-mis} in Definition~\ref{def:robust_detect_mis} and also \eqref{eq:l2_nonlinear}, the robust diagnosis filter has an optimization-based characterization where the objective function can be formulated to 
	minimize $\bar{N} ((1/m) \sum_{i=1}^{m} Q_{i}) \bar{N}^{\top}$ (average-cost viewpoint) or $\text{max}_{i \leq m} (\bar{N} Q_{i} \bar{N}^{\top})$ (worst-case viewpoint). We note that from computational perspective the average-cost is much more preferred. %
\end{Rem}

\begin{Thm}[Tractable quadratic programming characterization]\label{thm:qp_univ}
	Consider the polynomial matrices $H(q) = H_{0} + q \, H_{1}$ and $F(q) = F$ where $H_{0}, H_{1} \in \mathbb{R}^{n_{r} \times n_{\tilde{x}}}$ and $F \in \R^{n_{r} \times n_{f}}$ are constant matrices. The robust diagnosis filter introduced in \eqref{opt:robust-poly-mis} of Definition~\ref{def:robust_detect_mis} for the univariate attack can be obtained by solving the optimization program,
	\begin{align}\label{opt:quadratic_opt} 
	\min\limits_{\bar{N}}\quad & \bar{N} (\frac{1}{m} \sum_{i=1}^{m} Q_{i}) \bar{N}^{\top}  \nonumber \\
	\mbox{s.t.}\quad& \bar{N}\bar{H} = 0 \, , \\
	& \big\| \bar{N}\bar{F} \big \|_{\infty} \geq 1 
	\, , \nonumber
	\end{align}
	where $\|\cdot\|_{\infty}$ denotes the infinite vector norm, 
	and 
	\begin{align*}
	\bar{H} \Let \left[\begin{matrix} H_{0} & H_{1} & 0 & \cdots & 0 \\ 0 & H_{0} & H_{1} & 0 & \vdots \\ \vdots & 0  & \ddots & \ddots & 0 \\ 0 & \cdots & 0 & H_{0} & H_{1} \end{matrix}\right]. 
	\end{align*}
	Similar to the matrix $\bar{L}$ in \eqref{eq:l2_nonlinear}, $\bar{F}$ in \eqref{opt:quadratic_opt} is defined as $\bar{F} := diag [ F, \, F, \, \cdots, \, F]$. 
	Besides, a diagnosis filter from the program \eqref{opt:quadratic_opt} but simply adding the following linear constraint can have non-zero steady-state residual  
	that could approximate the attack value of $f$, 
	\begin{align}\label{eq:lem_uniatt_track}
	-a(1)^{-1}\displaystyle\sum\limits_{i=0}^{d_{N}} {N}_{i}F = 1.
	\end{align}
\end{Thm}

\begin{proof}
	The key step is to observe that in \eqref{opt:robust-poly-mis} we can rewrite $N(q)H(q) = \bar{N}\bar{H} [I ,\ qI ,\ \cdots ,\ q^{d_N+1}I ]^\top$, and $N(q)F(q) = \bar{N}\bar{F} [I ,\ qI ,\ \cdots ,\ q^{d_N}I ]^\top$. Note that in the scenario of univariate attack, one can directly rewrite the last constraint of \eqref{opt:robust-poly-mis} as $N(q)F(q) \neq 0$, and scale this inequality 
	to arrive at the one of $\| \bar{N}\bar{F} \|_{\infty} \geq 1$ in \eqref{opt:quadratic_opt}. 
	Next, if $\bar{N}\bar{H}=0$, the diagnosis filter becomes $r[k] = -a(q)^{-1}N(q)F(q)f[k] + a(q)^{-1}N(q)L(q)\varepsilon[k]$. If there is no output mismatch,  as a discrete-time signal, the steady-state value of the filter residual under the univariate attack would be $-a(q)^{-1}N(q)F(q)f|_{q=1}$. Note that $N(1)F(1) = \sum_{i=0}^{d_{N}} {N}_{i}F$.
	Thus when there exists output mismatch, with the linear constraint of 
	\eqref{eq:lem_uniatt_track}, 
	the residual in its steady-state value could approximate $f$.
\end{proof}

For the last constraint of \eqref{opt:quadratic_opt}, as indicated by \cite[Lemma~4.3]{Esfahani2016}, $\| \bar{N}\bar{F} \|_{\infty} \geq 1$ if and only if there exists a coordinate $j$ that $\bar{N}\bar{F}v_{j} \geq 1$ or $\bar{N}\bar{F}v_{j} \leq -1$. Here $v_{j} = [0,\cdots, 1, \cdots, 0]_{(d_{N}+1)}^{T}$ in which the only non-zero element of the vector is the $j$-th element. Thus, one can view \eqref{opt:quadratic_opt} as a family of $d_{N}+1$ standard QPs with linear constraints that for each QP the last constraint becomes $\bar{N}\bar{F}v_{j} \geq 1$ (or $\bar{N}\bar{F}v_{j} \leq -1$ since the set for the solutions of $\bar{N}$ in \eqref{opt:quadratic_opt} is symmetric). 
In addition, recall that the matrix $Q_{i}$ in the objective function is positive semi-definite, which implies that the resulted standard QPs are also convex, and hence tractable. Considering that the training phase of our own optimization-based characterization is a matter of matrix computation and the programs are convex QPs, we would say that our approach is not computational expensive, especially comparing with many machine-learning type methods.

\subsection{Extension to multivariate attack scenarios} \label{subsec:detect_multi}

Inspired by the techniques developed in \cite{Pan2020}, we further extend the preceding design of the robust diagnosis filter to the scenario of multivariate attacks. 

\begin{algorithm}[t!]
	\caption{Filter construction and validation in PF.}\label{alg:filterconstr_algorithm} 
	\begin{enumerate}
		\item[1)] \textbf{Pre-training}: Solve \eqref{opt:max-min-relax} for each $j \in \{1,\, \dots, \, 2d_N+2\}$. Check if there exists $\gamma^{\star}_{j} > 0$ and find the maximum. 
		\item[2)] \textbf{Training phase}: 
		\begin{enumerate}
			\item[(i)] For each instance $d_{i}$, 
			obtain $y_{p}$ and $y$ in the normal operations (without attacks). Calculate the matrices $\mathcal{E}_{i}$, $D_{i}$, and $Q_{i}$ according to \eqref{eq:mismatch_sig} - \eqref{eq:l2_nonlinear}. 
			\item[(ii)] For a number of $m$ instances of load disturbances, perform the same process in (i).
			\item[(iii)] Set the initial value of $\gamma_{j}$ 
			to be $\max_{\{j\le 2d_N+2\}} \gamma^{\star}_{j}$ from pre-training. Solve \eqref{opt:quadratic_opt} with the derived matrix $Q_{i}$. Tune the value of $\gamma_{j}$ until it reaches maximum.
		\end{enumerate}
		\item[3)] \textbf{Testing phase:} For another instance of $d_{i}$, 
		run the PF simulation where the FDI attacks are also launched. Run the resulted diagnosis filter for detection.
	\end{enumerate}
\end{algorithm}
\begin{Cor}[Robust diagnosis filter under multivariate attacks]\label{cor:qua_opt_transient} 
	Consider the diagnosis filter in Definition~\ref{def:robust_detect_mis} where the 
	set of multivariate attacks is defined as $\mathcal{F}  = \{f \in \mathbb{R}^{n_{f}} : f = F_{\mathrm b}^{\top} \alpha, \ \alpha \in \mathcal{A} \}$ in which $\mathcal{A} = \{ \alpha \in \mathbb{R}^{d} \ | \ A\alpha \geq b  \}$ (see Definition~\ref{def:robust_detect_mis} for the denotation of 
	these variables). 
	Given $j \in \{1, \, \dots, \, 2d_{N}+2 \}$, for each $j$, consider a family of the following quadratic programs, 
	\begin{align}\label{opt:qua_mul_trans}
	\min\limits_{\bar{N}, \lambda}\quad & \bar{N} (\frac{1}{m} \sum_{i=1}^{m} Q_{i}) \bar{N}^{\top} \, , \nonumber \\
	\mbox{s.t.}\quad& b^{\top}\lambda \geq \gamma_{j}, \tag{${\rm QP}_{j}$} \\
	& (-1)^{j}N_{\lfloor j/2 \rfloor}FF_{b} = \lambda^{\top}A \, , \nonumber\\ 
	& \bar{N}\bar{H} = 0, \ \lambda \geq 0 \, , \nonumber 
	\end{align}
	where ${\lfloor \cdot \rfloor}$ is the ceiling function that maps the argument to the least integer. Then, the best solution of the programs~\eqref{opt:qua_mul_trans} among $j \in \{1,\, \dots, \, 2d_N+2\}$ solve the problem \eqref{opt:robust-poly-mis} in Definition~\ref{def:robust_detect_mis} for the multivariate attack scenario.
\end{Cor}

\begin{proof}
	In the scenario of multivariate attacks, the two constraints in \eqref{opt:robust-poly-mis} can be characterized by the 
	maximin program,
	\begin{align}\label{opt:robust_maximin}
	\gamma^\star \Let \max\limits_{ \bar{N} \in \mathcal{N} } \ \min\limits_{\alpha \in \mathcal{A} } \ \big\{ \mathcal{J}(\bar{N},\alpha) \big\}, 
	\end{align}
	where the set $\mathcal{N} := \big\{\bar{N} \in \mathbb{R}^{(d_{N}+1)n_{r}} \ | \ \bar{N}\bar{H} = 0 \big\}$. The source of the cost function $\mathcal{J}(\bar{N},\alpha)$ is referred to \cite[Section IV.B]{Pan2020}. 
	Then, according to \cite[Theorem IV.3]{Pan2020}, we know that the 
	maximin program \eqref{opt:robust_maximin} can be reformulated and relaxed to a set of linear programs (LPs), 
	\begin{align}\label{opt:max-min-relax}
	\gamma^{\star}_{j} := \max\limits_{\bar{N}, \, \lambda} \quad&  b^{\top}\lambda \nonumber \\
	\mbox{s.t.} \quad& (-1)^{j}N_{\lfloor j/2 \rfloor}F\basis = \lambda^{\top}A, \tag{${\rm LP}_{j}$}\\
	& \bar{N}\bar{H} = 0, \ \lambda \geq 0 \, , \nonumber 
	\end{align}
	Namely, the solution to the program~\eqref{opt:max-min-relax} is a feasible solution to the maximin program \eqref{opt:robust_maximin},  
	and it satisfies $\max_{\{j\le 2d_N+2\}} \gamma^{\star}_{j} \le \gamma^\star$. Then it is easy to obtain the finite \eqref{opt:qua_mul_trans} for the multivariate attack scenario. We conclude the proof by noting that if there exist a $\gamma_{j}^\star > 0$, then a resulted filter from \eqref{opt:max-min-relax} could detect all the 
	plausible multivariate attacks in the set $\mathcal{F}$. 
\end{proof}

From Corollary~\ref{cor:qua_opt_transient}, we can see that for any $j \in \{1,\, \dots, \, 2d_N+2\}$, if one can find a $\gamma_{j} > 0$ that \eqref{opt:qua_mul_trans} is still feasible, then the solution to \ref{opt:qua_mul_trans} offers a robust diagnosis filter in the type of Definition~\ref{def:robust_detect_mis} for multivariate attacks. It is worth mentioning that the obtained \eqref{opt:qua_mul_trans} in Corollary~\ref{cor:qua_opt_transient} 
is also a standard convex QP. 

For a better illustration, Algorithm~\ref{alg:filterconstr_algorithm} concludes the filter construction and validation process for an implementation in the PF simulator under the scenario of multivariate attacks. In the ``pre-training'', one needs to solve~\eqref{opt:max-min-relax} for each $j$ to see if there exists $\gamma^{\star}_{j} > 0$. If yes, next in the ``training phase'', $\mathcal{E}_{i}$, $D_{i}$ and $Q_{i}$ can be obtained from the simulation data in the normal operations for each disturbance signature $d_{i}$. Then the program~\eqref{opt:qua_mul_trans} needs to be solved and the resulted robust diagnosis filter can be ``tested''. 
We would like to highlight that, the robust diagnosis filter from \eqref{opt:qua_mul_trans} does not necessarily enforce a non-zero steady-state residual in multivariate attack scenarios. Regarding its steady-state behavior, the program \eqref{opt:robust_maximin} can be modified into $\mu^\star \Let \max_{\{\bar{N} \in \mathcal{N}\}} \ \min_{\{\alpha \in \mathcal{A}\}} \ | \bar{N} \bar{F} \alpha |$ which has an exact convex reformulation. Then a similar treatment as the one in Corollary~\ref{cor:qua_opt_transient} can be deployed. 

In the end, we summarize the filter features under univariate or multivariate attacks and in the pure model-based or our proposed data-assisted model-based methods, in Table~{\ref{tab:sum_work}}.

\color{red}
\renewcommand\arraystretch{1.5}
\begin{table}[t]
	\centering
	\caption{A summary of filter features in different scenarios.}
	\label{tab:sum_work}
	\begin{tabular}{|p{0.01\textwidth}|p{0.06\textwidth}|p{0.06\textwidth}|p{0.06\textwidth}|p{0.06\textwidth}|}
		\hline
		\multicolumn{1}{|c|}{\multirow{2}{*}{}} & \multicolumn{2}{c|}{\textbf{Attack scenario}} & \multicolumn{2}{c|}{\textbf{Model-based detection}} \\\cline{2-5}
		\multicolumn{1}{|c|}{}  & \multicolumn{1}{c|}{univariate}  & \multicolumn{1}{c|}{multivariate} & \multicolumn{1}{c|}{pure} & \multicolumn{1}{c|}{data-assisted} \\ 
		\hline
		\multicolumn{1}{|l|}{\textbf{Filter features}} &  \multicolumn{1}{c|}{$n_{f} = 1$} &  \multicolumn{1}{c|}{$n_{f} > 1$} &  \multicolumn{1}{c|}{\begin{tabular}[c]{@{}c@{}}$Q_{i} = 0$, \\ in \cite{Pan2020}\end{tabular}} & \multicolumn{1}{c|}{\begin{tabular}[c]{@{}c@{}}$Q_{i}$ from \eqref{eq:l2_nonlinear}, \\this study \end{tabular}} \\
		\hline
	\end{tabular}
\end{table}
\color{black}

%

\section{Numerical Results} \label{sec:results}

\subsection{Test system and robust filter description} \label{subsec:test_sys}

To validate the effectiveness of our data-assisted model-based diagnosis filter, it has been implemented in the high-fidelity simulator of PF to detect FDI attacks on the ACE signal of AGC in the three-area 39-bus system. The system parameters of the 
abstract linear model are referred to \cite{bevrani2008}, and the specifications of the model in PF are available at \cite{powerfactory39}. Following Algorithm~\ref{alg:filterconstr_algorithm}, to obtain output mismatch signatures, we run the simulations to obtain $y_{p}$ and $y$ with the same input $d$ of load disturbances in normal operations. The adjustable degree of the residual generator is set to $d_{N} = 3$ which is much less than the order of the linearized model (it is a 19-order model of \eqref{eq:linearSP}); the scalar polynomial $a(q)$ is set $a(q) = (q - p)^{d_{N}}/{{(1-p)}^{d_{N}}}$ where $p$ is a user-defined variable, acting as the pole of $R_{\varepsilon}(q)$. We let $T_s = 0.5 \, \mathrm{s}$ such that the simulation results from PF are sampled, and thus for a simulation time $t_{s} = 10\, \mathrm{s}$, $T = 20$ in \eqref{eq:mismatch_sig}. 
We also perform a comparison study for the two methods: our data-assisted model-based filter and the pure model-based one.

\subsection{Simulation Results}

The first simulation considers the univariate attack 
where an attacker has manipulated the power exchange between Area 1 and Area 2 from $t = 30 \, \mathrm{s}$ in the horizon of $60 \, \mathrm{s}$. To  
challenge the filter, the disturbances are modeled as stochastic load patterns: load variation of Load 4 in Area 1 is a random zero-mean Gaussian signal. A number of $m=100$ load disturbance instances are generated for the ``training phase'' where for each of them, a simulation with $t_{s} = 10 \, \mathrm{s}$ is conducted to obtain output mismatch signatures. The design variables $\bar{N}$ are derived from 
Theorem~\ref{thm:qp_univ}. To compare, a pure model-based filter 
is also derived by letting $Q_{i} = 0$ (see Table~\ref{tab:sum_work}), which can be transformed into finite LPs. The simulation results are referred to Appendix B. 
We can see that our data-assisted model-based filter has significant improvements in the regards of mitigating the effect from output mismatch on the residual, comparing with the pure model-based one. Besides, 
it can 
track 
attack value 
through its steady-state residual. The pure model-based filter fails by triggering false alarms, when applied 
to the PF simulator. This can be expected since it utilizes the 
abstract model information only and suits for such given model of \eqref{eq:linearSP}. 


\begin{figure}[t!]
	\centering
	\begin{subfigure}[t]{0.47\textwidth}
		\centering
		\includegraphics[scale=0.44]{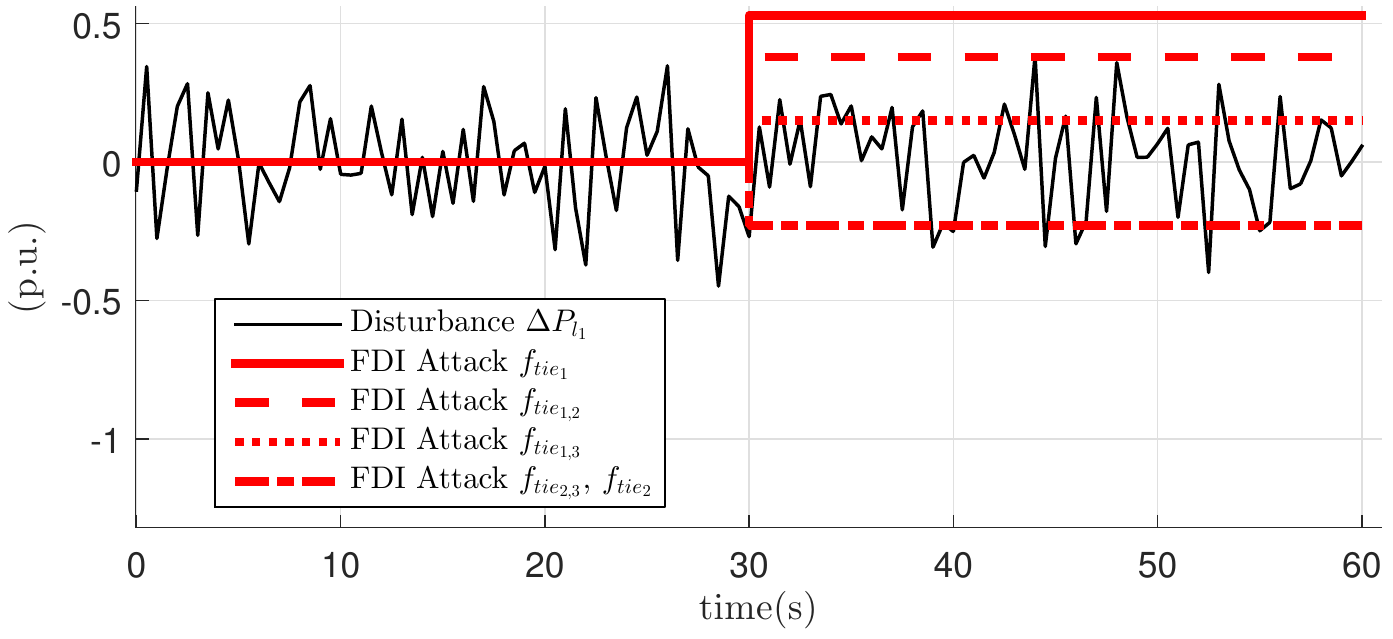}
		\caption{Load disturbance and multivariate attack.}\label{subfig51:mastd}
	\end{subfigure}
    ~
    \begin{subfigure}[t]{0.47\textwidth}
    	\centering
    	\includegraphics[scale=0.445]{fig1ad_mastl_p08_2609.pdf}
    	\caption{Load disturbance and multivariate attack.}\label{subfig61:mastd}
    \end{subfigure}
	\\ 
	\begin{subfigure}[t]{0.47\textwidth}
		\centering
		\includegraphics[scale=0.44]{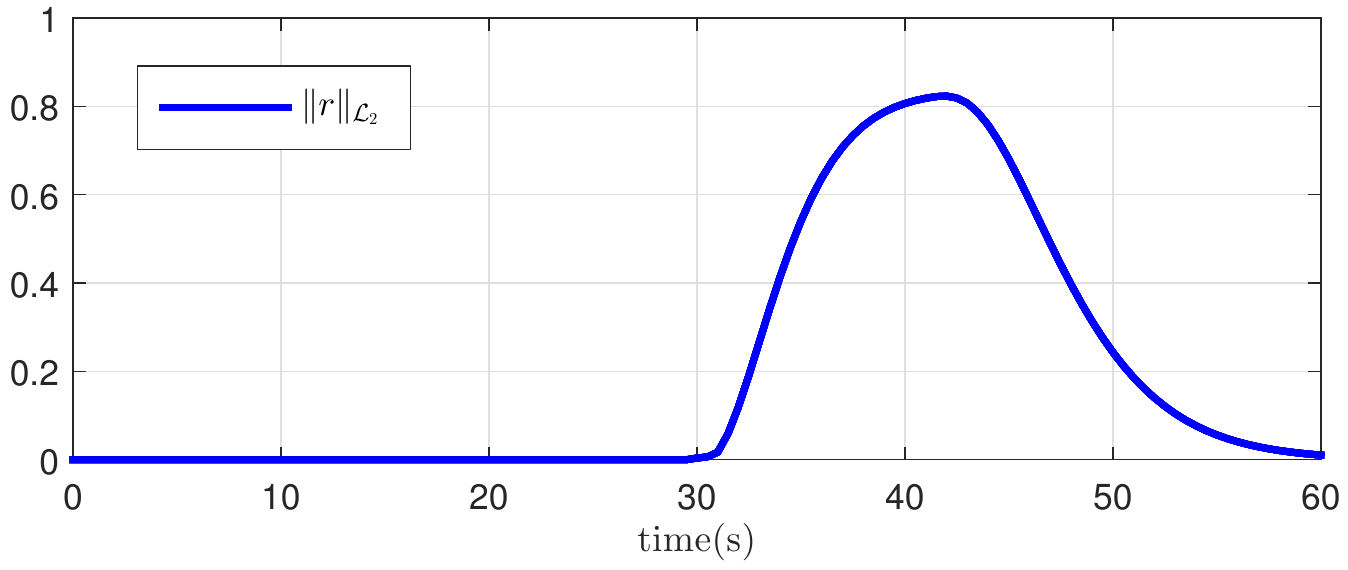}
		\caption{Energy of residual 
			without output mismatch.}\label{subfig52:ro}
	\end{subfigure}
    ~
    \begin{subfigure}[t]{0.47\textwidth}
    	\centering
    	\includegraphics[scale=0.445]{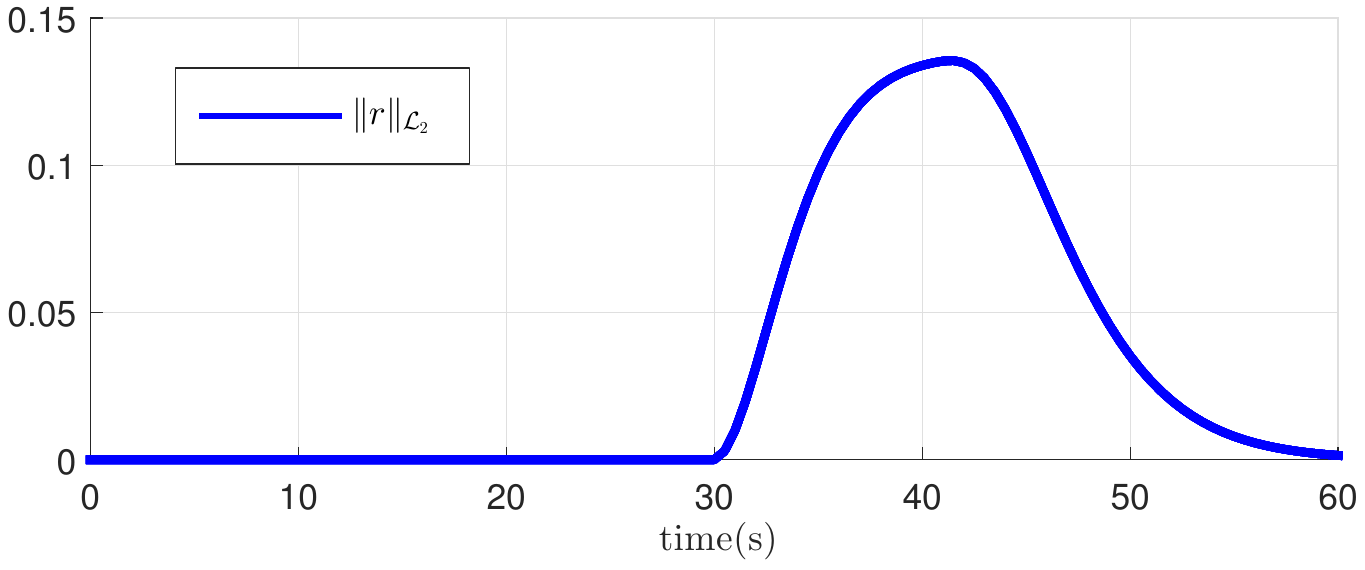}
    	\caption{Energy of residual 
    		without output mismatch.} \label{subfig62:rrob}
    \end{subfigure}
	\\
	\begin{subfigure}[t]{0.47\textwidth}
		\centering
		\includegraphics[scale=0.455]{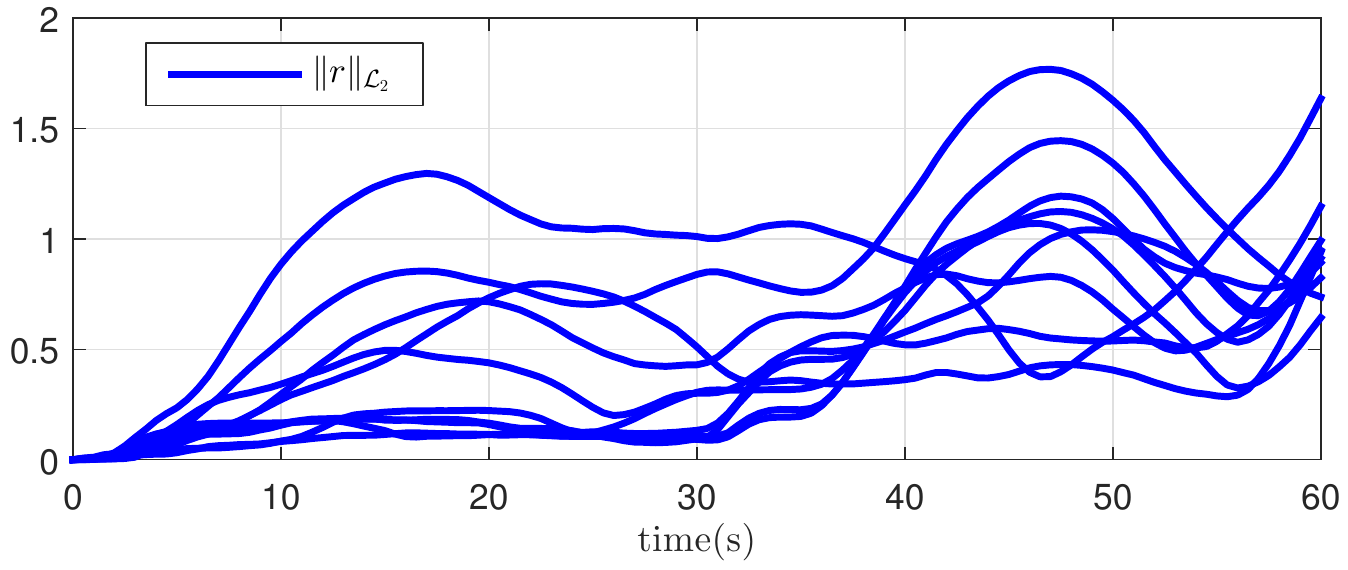}
		\caption{Energy of residual 
			with output mismatch. 
		}\label{subfig:ron210stl}
	\end{subfigure}
    ~
    \begin{subfigure}[t]{0.47\textwidth}
    	\centering
    	\includegraphics[scale=0.455]{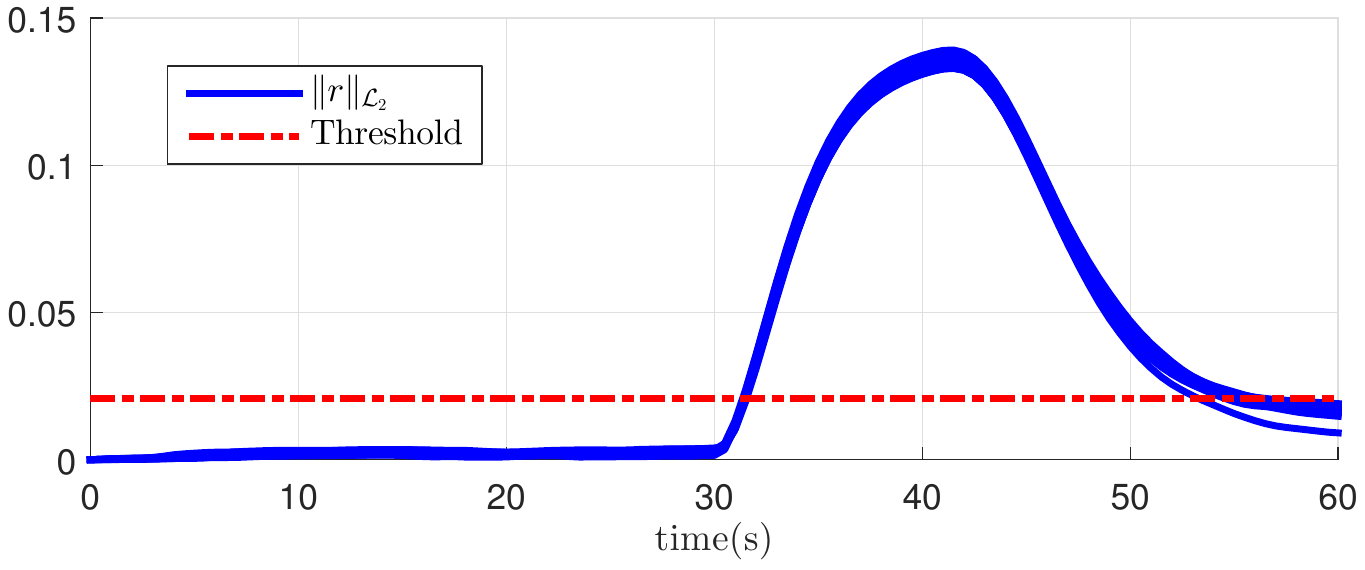}
    	\caption{Energy of residual 
    		with output mismatch. 
    	}\label{subfig:rrobn210stl}
    \end{subfigure}
	\caption{The left: pure model-based filter in \cite{Pan2020}. It is derived by letting $Q_{i} = 0$ in Corollary~\ref{cor:qua_opt_transient}, and it is essentially a set of LPs.
	The right: data-assisted model-based filter. It is derived by solving \eqref{opt:qua_mul_trans} in Corollary~\ref{cor:qua_opt_transient} where $Q_{i}$ is from \eqref{eq:l2_nonlinear}} 
	\label{fig:resfig_puremodel_ma} %
\end{figure}

In the second simulation, we move to the scenario of multivariate attacks. There are 5 
power exchanges between 
areas that are attacked, 
and correspondingly there exist 3 basis vectors in the spanning set $\mathcal{F}$: $f_{1}= [0.1 \ \ 0 \ \ 0.1 \ \ 0 \ \ 0]^T$, $f_{2}= [0.1 \ \ 0.15 \ \ 0.25 \ \ 0 \ \ 0]^T$, $f_{3}= [0 \ \ 0 \ \ 0 \ \ 0.1 \ \ 0.1]^T$ (all in $\mathrm{p.u.}$). Besides, for the set of disruptive multivariate attacks, the parameters are set to $A=\mathbf{1}^\top$ and $b=1.5$ in $\mathcal{A}$. We refer to \cite[Section V]{Pan2020} for the specification of these values. 
Following Algorithm~\ref{alg:filterconstr_algorithm}, the program \eqref{opt:max-min-relax} is solved. The optimal value achieves maximum for $j=2$ that $\gamma^{\star}_{2} = 300$, which implies that a diagnosis filter of our approach 
could be obtained. Next, in the ``training phase'', a number of $m=100$ load disturbance instances are randomly generated. The program \eqref{opt:qua_mul_trans} in Corollary~\ref{cor:qua_opt_transient} is solved for the 
filter design. For the derived $\bar{N}$, the multivariate attack coordinate vector $\alpha$ is obtained by solving the inner minimization of the program \eqref{opt:robust_maximin}. In the ``test phase'', simulations in PF 
are conducted that several realizations of load disturbances have been implemented and the multivariate attacks with $\alpha$ have been launched. The performance of the two filters (the filter of our approach and the pure model-based filter from 
\cite{Pan2020}) is validated with two sets of outputs: one from the 
abstract linear model \eqref{eq:linearSP} 
(i.e., without output mismatch, $\varepsilon \equiv 0$) and another one from the PF 
simulations (i.e., with output mismatch, $\varepsilon \neq 0$). 
Figure~\ref{fig:resfig_puremodel_ma} 
shows the simulation result for both diagnosis filters. We can see that both filters succeed for the case $\varepsilon \equiv 0$, according to Figure~\ref{subfig52:ro} and~\ref{subfig62:rrob}. However, from Figure~\ref{subfig:ron210stl} and~\ref{subfig:rrobn210stl}, when there exists output (model) mismatch, our 
data-assisted model-based filter still works effectively, while the pure model-based filter in \cite{Pan2020} totally fails by triggering false alarms. 
Note that Figure~\ref{fig:resfig_puremodel_ma} 
depicts the ``energy'' of the residual signals for the last $10 \mathrm{s}$ under 10 load disturbance instances, namely $\| r \|_{\mathcal{L}_{2}}[\cdot]$. In Figure~\ref{subfig:rrobn210stl}, the threshold is set to $\tau^{\star} + 0.025$ where the square of $\tau^{\star}$ equals to the maximum value of $\bar{N}Q_{i}\bar{N}$ in the 100 training instances ($i \in \{1, \, \cdots, \, 100\}$), and the added value is to avoid possible false alarms according to \cite{MohajerinEsfahani2015}. Then a multivariate attack is said to be detected when the value of $\| r \|_{\mathcal{L}_{2}}[\cdot]$ is beyond this threshold; we see successful detections by our proposed method in Figure~\ref{subfig:rrobn210stl}, while the pure model-based one totally fails in Figure~\ref{subfig:ron210stl}. 
In the end, note that when looking into the steady-state behavior of the filter, 
it turns out that $\mu^\star = 0$, which indicates that the optimal multivariate attack in this case is a stealthy attack in the long-term horizon, with or without considering the output mismatch effect. However, one can still detect such attacks with a non-zero transient residual, as shown in Figure~\ref{subfig:rrobn210stl}. In conclusion, these simulation results 
validate the effectiveness of our proposed 
solution. 

\subsection{Discussions} \label{subsec:discuss}

\subsubsection{Applicability of the proposed solution}
It can be extended to other systems and anomaly scenarios 
if the anomalies are acting as exogenous inputs (see $f$ in \eqref{eq:linearSP}). 
Another instance can be the small-signal dynamics model of power systems under anomalies such as cyber attacks or bad data on the measurements (e.g., terminal voltage/current phasor of each generator). The mathematical description of the model can be linearized by considering a small perturbation over an 
operating point, resulting an abstract 
model 
\cite[Section III]{Taha2016}. 

\subsubsection{Comparison with other pure model-based or data-driven detectors}
As illustrated in simulation results, it can be expected that model mismatch would affect the diagnosis of other pure model-based detectors like observers when applied to the high-fidelity simulator. 
Regarding data-driven approaches, we admit that they may also be able to detect the anomalies 
of this article, while the data training stage may require a high computational cost, comparing with our proposed tractable convex optimization-based characterization.

\section{Conclusion} \label{sec:conclusion}

We have proposed a feasible solution to the problem that arises from applying scalable model-based anomaly detectors in practice: there always exists mismatch 
between 
the picked abstract model and the detailed one in the high-fidelity simulator (or the real electric power system). In our final tractable reformulation, the model-based information introduces feasible sets, and the simulation data forms the objective function to minimize the model mismatch effect, which could bridge the model-based and data-driven approaches.

\setcounter{section}{1}
\section*{Appendix A: System Modeling Details for Section II} \label{sec:appendix}

\subsection{AGC system and FDI attack in PowerFactory}

\begin{figure}[t!p]
	\centering
	\includegraphics[scale=0.089]{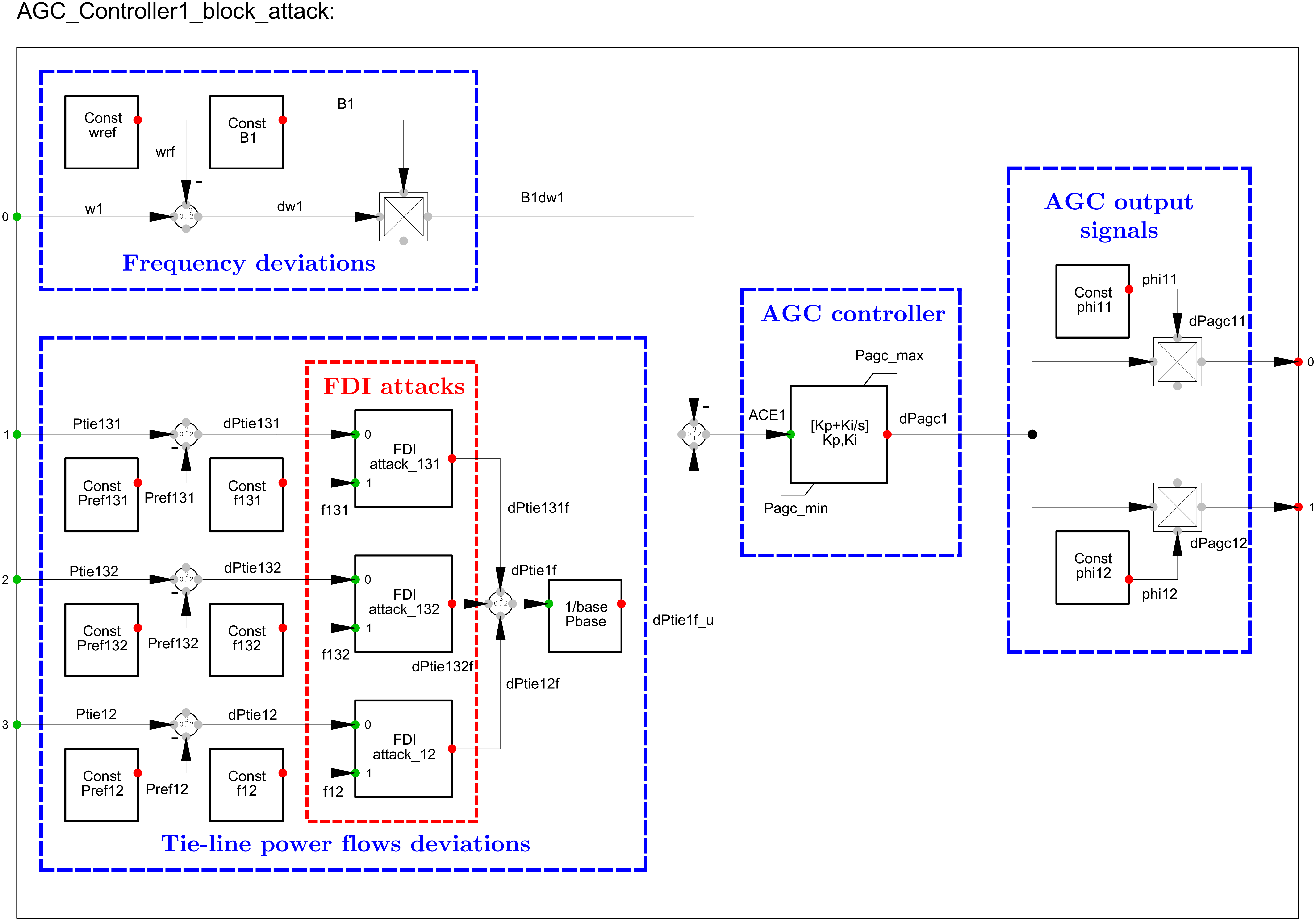}
	\caption{The block definition of AGC in high-fidelity simulator PF.}
	\label{fig:AGC_ctrl_pf}
\end{figure}

In a high-fidelity simulator like 
PF, a detailed simulation of power system dynamics can be carried out that all the details of generators (e.g., 
synchronous machine, exciter, voltage regulator, turbine-governor unit), electric network and also controllers such as 
AGC can be included. 
Figure~1 in the main body of the article depicts the three-area IEEE 39-bus system with AGC functions in the PF simulator. In the simulations, the dynamic generator model can consist of a synchronous machine, along with voltage regulator for the excitation system in the type of IEEE Type 1, and turbine-governor model in the type of IEEE Type G1 (steam turbine) or IEEE Type G3 (hydro turbine). 

The generators in red diagrams of Figure~1 are also 
participating AGC control of each area. The AGC control loop can be developed by PF's own modeling language - {\em DIgSILENT Simulation Language (DSL)}. 
For instance, Figure~\ref{fig:AGC_ctrl_pf} illustrates the \emph{block definition} of AGC in Area 1, which has several sub-blocks that collect frequencies and power exchanges between areas as the control inputs and perform AGC function to calculate signals for power settings of the participated generators in Area 1. Moreover, we build another block definition for the FDI attack on the frequencies and power exchanges which are parts of the ACE signal. 


We have built the high-fidelity simulation model in the simulator PF for the 39-bus system equipped with AGC functions. 
The {\em composite frame} of AGC 
builds the connections between the inputs and outputs of the AGC model elements. Then the AGC {\em block definitions} for all areas can be created. 
For Figure~\ref{fig:AGC_ctrl_pf} of block definition of AGC in Area 1, the four sub-blocks include 

\begin{itemize}
	\item {\em frequency deviations} block where the frequency deviations in $\mathrm{p.u.}$ multiplied by 
	a frequency bias factor are calculated;
	\item {\em tie-line power flows deviations} block which computes the tie-line power flow deviations (normalized in $\mathrm{p.u.}$) on the side of Area 1 for the power part of ACE; \item {\em AGC controller} block which performs the ACE calculation and the integral action to generate 
	the tuning signal for power settings of participated generators. To be noted, the saturation effects are considered that the limits of $P_{agc}^{min}$ and $P_{agc}^{max}$ are added for 
	this tuning signal. 
	\item {\em AGC output signals} block where the tuning signals for the participated generators in Area 1 for AGC are calculated based on each generator's	participating factor. %
\end{itemize}

The above block definitions are modeled using the {\em Standard Macros} of PowerFactory global {\em Library}. Moreover, in Figure~\ref{fig:AGC_ctrl_pf}, another block definition (in red diagram) corresponds to the FDI attack model for the study of this article, 
\begin{itemize}
	\item {\em FDI attacks} block where the FDI attack is implemented. Each block captures the feature of the stationary FDI attack, i.e., the attack occurs as a constant bias injection~($f[k]=f$) on measurements at a specific time step, 
	and it remains unchanged since then. This block can add an ``false'' injection 
	into the existing signal. One can specify the occurrence time and the attack values. This block definition is achieved by using the {\em digexfun} interface. With {\em digexfun}, we can define a specific DSL function (in C++) and create a dynamic link library {\em digexfun\_{*}.dll} that the PF can load.
\end{itemize}

\subsection{Parameters of the abstract linear model} \label{subsec:app_parameters}

\begin{figure}[t!p]
	\centering
	\includegraphics[scale=0.56]{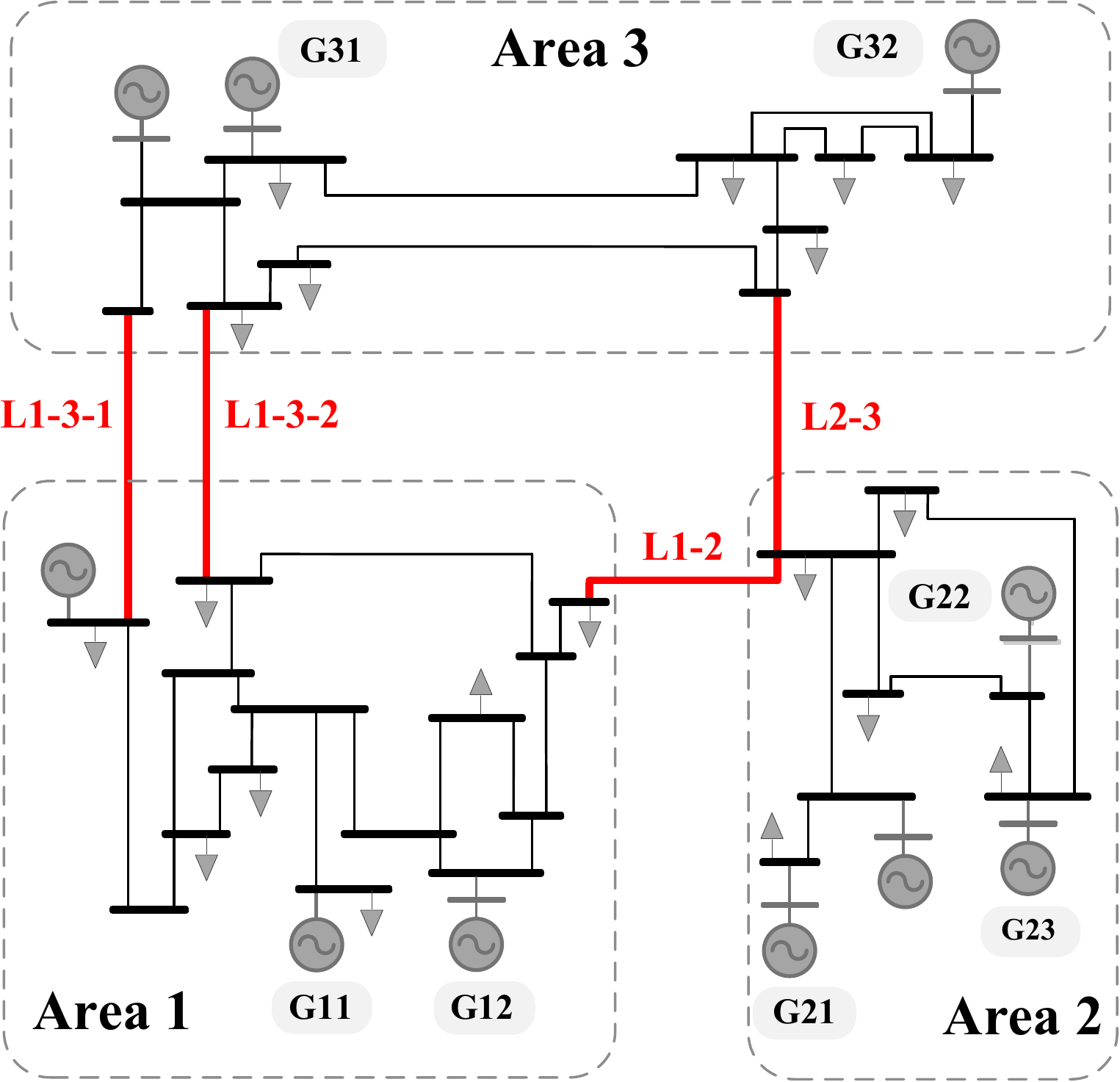}
	\caption{The diagram of the three-area system \cite{Pan2020}.} 
	\label{fig:39bus}
\end{figure}

The abstract linear model (2) in Section II is built through the following process. For AGC analysis, as mentioned in Section II, we are interested in collective performance
of all generators, and thus we can rely on certain levels of abstraction that simplify some elements of the initial complex model 
and utilize the possible decoupling between control loops. 
Then the mathematical description of a AGC system in Area $i$ can be presented in the linear formulation where each area of a power system is represented by a simplified model comprised of equivalent governors, turbines and generators, 
\begin{equation}\label{eq:AGC_model}
\left\{
\begin{aligned}
& \Delta \dot{\omega}_{i} = \frac{1}{2H_{i}}(\Delta {P}_{m_{i}}-\Delta P_{tie_{i}} -\Delta P_{d_{i}}-D_{i} \Delta \omega_{i}), \\
&  \Delta {P}_{m_{i}} = \sum_{g=1}^{G_{i}} \Delta {P}_{m_{ig}}, \ \Delta P_{tie_{i}} = \sum_{j \in \mathcal{M}_{i}} \Delta P_{tie_{ij}}, \\
& \Delta \dot{P}_{m_{ig}} = -\frac{1}{T_{ch_{ig}}}(\Delta P_{m_{ig}} + \frac{1}{S_{ig}} \Delta \omega_{i} - \phi_{ig}\Delta P_{agc_{i}}), \\
& \Delta \dot{P}_{tie_{ij}} =  T_{ij} (\Delta \omega_{i} - \Delta \omega_{j}), \\
& ACE_{i} = {\beta}_{i} \Delta \omega_{i} + \sum_{j \in \mathcal{M}_{i}} \Delta P_{tie_{ij}}, \\
& \Delta \dot{P}_{agc_{i}} =  - K_{I_{i}} ACE_{i}, %
\end{aligned}
\right.
\end{equation}
where $H_{i}$ is the equivalent inertia constant of Area $i$, 
$D_{i}$ is the damping coefficient, $\Delta {P}_{m_{i}}$ and $\Delta P_{tie_{i}}$ are the total generated power 
in Area $i$ and the total tie-line power exchanges from Area $i$, and $\Delta P_{d_{i}}$ denotes load disturbances. 
The term $G_{i}$ is the number of participated generators in Area $i$, and $\mathcal{M}_{i}$ is the set of areas that connect to Area $i$ in the multi-area system. Besides, $T_{ch_{ig}}$ is the governor-turbine's time constant, $S_{ig}$ is the droop coefficient, and $T_{ij}$ is the synchronizing parameter between Area $i$ and $j$. 

In the AGC loop, note that $\Delta P_{agc_{i}}$ in \eqref{eq:AGC_model} is the signal from the AGC controller for the participated generators to track the load changes, and $\phi_{i,g}$ is the participating factor, i.e., $\sum_{g=1}^{G_{i}} \phi_{i,g} =1$. After receiving the information of frequency and 
power exchange,  
the ACE signal is computed for an integral action where $\beta_{i}$ is the frequency bias and $K_{I_{i}}$ represents the integral gain. Based on the equations in \eqref{eq:AGC_model}, the abstract linear model of Area $i$ can be presented as
\begin{subequations}\label{eq:sp_areai}
	\begin{align}
	& \dot{\tilde{x}}_{i}(t) = {A}_{ii} \tilde{x}_{i}(t) + {B}_{id}{d}_{i}(t) + \sum_{j \in \mathcal{M}_{i}} {A}_{ij} \tilde{x}_{j}(t), \\
	& {y}_{i}(t) = {C}_{i}\tilde{x}_{i}(t),
	\end{align}
\end{subequations}
where $\tilde{x}_{i} := \big[\{\Delta P_{tie_{ij}}\}_{j \in \mathcal{M}_{i}}, \Delta \omega_{i}, \{\Delta {P}_{m_{ig}}\}_{1:G_{i}}, \Delta P_{agc_{i}}\big]^{\top}$ is the state vector  
that consists of tie-line power exchange, area frequency, generator output, and also AGC control signal of the close-loop system; ${d}_{i} := [\Delta P_{d_{i}}]^{\top}$ denotes load disturbances. Besides, in \eqref{eq:sp_areai}, ${A}_{ii}$ is the system matrix of Area $i$, ${A}_{ij}$ is a matrix whose only non-zero element is $-T_{ij}$ in row 1 or 2 and column 3,  ${B}_{id}$ is the matrix for load variations. We assume an output model with high redundancy that the measurements of each tie-line power ($\Delta P_{tie_{ij}}$) and the total tie-lines' power ($\Delta P_{tie_{i}}$), the frequency $\Delta \omega_{i}$, each generator output ($\Delta {P}_{m_{ig}}$) and the total generated power ($\Delta {P}_{m_{i}}$), and the AGC control signal $\Delta P_{agc_{i}}$ are all available. Then ${y}_{i}$ is the output of Area $i$ and ${C}_{i}$ is the output matrix with full column rank. As noted earlier, vulnerabilities within the 
communication channels for frequencies and power exchanges as parts of the ACE signal
may allow FDI attacks. For instance, if an attack can manipulate the output of 
one of tie-line power exchanges from Area $i$, say $f_{{tie}_i}$, then the ACE signal in \eqref{eq:AGC_model} would be corrupted into
\begin{equation}\label{eq:ace_f}
ACE_{i} = {\beta}_{i} \Delta \omega_{i} + (\sum_{j \in \mathcal{M}_{i}} \Delta P_{tie_{ij}} + f_{{tie}_i}), \\
\end{equation}
which implies that corruptions on the system output would affect the dynamics of controllers and consequently the involved physical system. 

Figure~\ref{fig:39bus} depicts the diagram of the three-area system used in the mathematical description of the abstract linear model. For the three-area AGC system, the vectors $\tilde{x}$, $d$ and $f$ in the continuous-time model (2) of Section II can be expressed as
\begin{gather}\label{eq:spx_39_x&d&f}
\tilde{x}= \left[\begin{matrix} \tilde{x}_{1}^{\top} & \tilde{x}_{2}^{\top} & \tilde{x}_{3}^{\top} \end{matrix}\right]^{\top}, \nonumber\\
d = \left[\begin{matrix} \Delta P_{d_{1}} & \Delta P_{d_{2}} & \Delta P_{d_{3}}\end{matrix}\right]^\top, \quad f= \left[\begin{matrix} f_{1}^\top & f_{2}^\top & f_{3}^\top \end{matrix}\right]^\top. \nonumber
\end{gather} 
In (2), ${A}_{c}$ is the closed-loop system matrix, ${B}_{c,d}$, ${B}_{c,f}$ are constant matrices, and these matrices can be described by
\begin{gather}\label{eq:spx_ABdBf}
{A}_{c} = \left[\begin{matrix} A_{11} & A_{12} & A_{13} \\ A_{21} & A_{22} & A_{23} \\  A_{31} & A_{32} & A_{33} \end{matrix}\right] \, , \nonumber \\
{B}_{c,d} = \mbox{diag} \left[\begin{matrix}B_{1d}, \ B_{2d}, \ B_{3d} \end{matrix}\right], \nonumber \\
{B}_{c,f} = \mbox{diag} \left[\begin{matrix}B_{1f}, \ B_{2f}, \ B_{3f} \end{matrix}\right] \,. \nonumber
\end{gather}
Similarly, for the output equation of the three-area system, the output vector $y$, the output matrix $C$ and the matrix $D_{f}$ which quantifies the places of vulnerable measurements become
\begin{gather}\label{eq:spx_39_y&c&df}
y= \left[\begin{matrix} y_{1}^\top & y_{2}^\top & y_{3}^\top \end{matrix}\right]^\top, \nonumber\\
C = \mbox{diag} \left[\begin{matrix}C_{1}, \ C_{2}, \ C_{3} \end{matrix}\right], \ D_{f} = \mbox{diag} \left[\begin{matrix}D_{1f}, \ D_{2f}, \ D_{3f} \end{matrix}\right]. \nonumber
\end{gather}

Next, we illustrate the formulations of the involved matrices for each area. Let us take the model description of Area 1 in the three-area system as an instance. More details can be found in \cite{Pan2020}. As seen from Figure~\ref{fig:39bus}, Area 1 has two generators (G 11, G 12) participating in the AGC operation and it is connected with Area 2 and Area 3 through the transmission lines called tie-lines (L1-3-1, L1-3-2, L1-2). Then we can have
\begin{equation} \label{eq:spx_B1d}
B_{1d} = \left[\begin{matrix} 0 & 0 & -\frac{1}{2H_{1}} & 0 & 0 & 0 \end{matrix}\right]^{\top}, \nonumber
\end{equation}
\begin{gather}\label{eq:spx_A11}
A_{11} = \small\left[\begin{matrix} 0 & 0 & T_{12} & 0 & 0 & 0 \\ 0 & 0 & T_{13} & 0 & 0 & 0 \\ -\frac{1}{2H_{1}} & -\frac{1}{2H_{1}} & -\frac{D_{1}}{2H_{1}} & \frac{1}{2H_{1}} & \frac{1}{2H_{1}} & 0 \\ 0 & 0 & -\frac{1}{T_{ch_{11}}S_{11}} & -\frac{1}{T_{ch_{11}}} & 0 & \frac{\phi_{11}}{T_{ch_{11}}} \\ 0 &  0 & -\frac{1}{T_{ch_{12}}S_{12}} & 0 & -\frac{1}{T_{ch_{12}}} & \frac{\phi_{12}}{T_{ch_{12}}} \\ -K_{I_{1}} & -K_{I_{1}} & -K_{I_{1}}B_{1} & 0 & 0 & 0 \end{matrix}\small\right]. \nonumber
\end{gather}

As we have assumed a measurement model with high redundancy, the matrix $C_{1}$ for Area 1 becomes
\begin{equation} \label{eq:spy_c1}
C_{1} = \left[\begin{matrix} 1 & 0 & 0 & 0 & 0 & 0 \\ 0 & 1 & 0 & 0 & 0 & 0 \\ 0 & 0 & 1 & 0 & 0 & 0 \\ 0 & 0 & 0 & 1 & 0 & 0 \\ 0 & 0 & 0 & 0 & 1 & 0 \\ 0 & 0 & 0 & 0 & 0 & 1 \\ 1 & 1 & 0 & 0 & 0 & 0 \\ 0 & 0 & 0 & 1 & 1 & 0\end{matrix}\right]^{\top}. \nonumber
\end{equation}

For the attack scenario, notably in practise, the tie-line power exchanges are usually more vulnerable to cyber attacks, comparing with frequency measurements (e.g., the anomalies in frequency can be easily detected by comparing the corrupted reading with the normal one). Thus to illustrate the formulations of matrices $D_{1f}$ and $B_{1f}$, let us assume that the vulnerable measurements in Area 1 to cyber attacks are the ones of power exchanges $\Delta P_{tie_{12}}$, $\Delta P_{tie_{13}}$ and $\Delta P_{tie_{1}}$. 
Then the parameters regarding such multivariate attacks are
\begin{equation} \label{eq:spy_f1}
f_{1} = \left[\begin{matrix} f_{tie_{12}} & f_{tie_{13}} & f_{tie_{1}}\end{matrix}\right]^{\top}, \nonumber
\end{equation}
\begin{equation} \label{eq:spy_D1f}
D_{1f} = \left[\begin{matrix} 1 & 0 & 0 & 0 & 0 & 0 & 0 & 0 \\ 0 & 1 & 0 & 0 & 0 & 0 & 0 & 0 \\ 0 & 0 & 0 & 0 & 0 & 0 & 1 & 0 \end{matrix}\right]^{\top}, \nonumber
\end{equation}
\begin{equation}\label{eq:spx_B1f}
B_{1f} = \left[\begin{matrix} 0 & 0 & 0 & 0 & 0 & -k_{1} \\ 0 & 0 & 0 & 0 & 0 & -k_{1} \\ 0 & 0 & 0 & 0 & 0 & 0 \end{matrix}\right]^{\top}. \nonumber
\end{equation}
\begin{figure}[t!p]
	\centering
	\begin{subfigure}[t]{0.46\textwidth}
		\centering
		\includegraphics[scale=0.44]{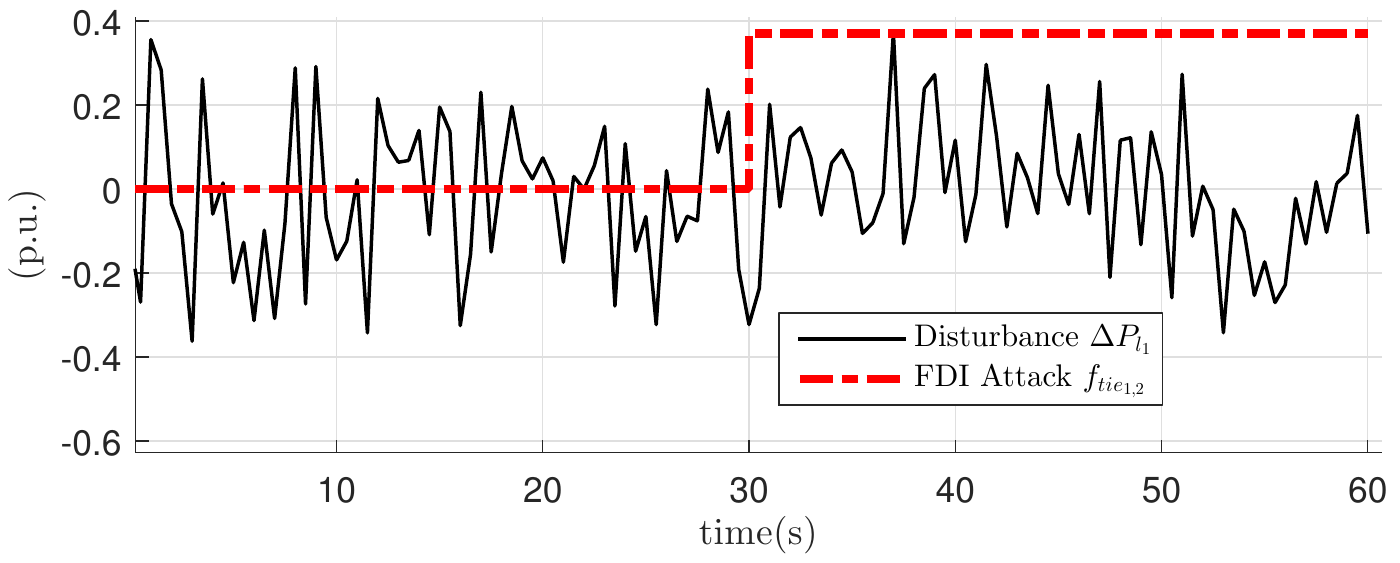}
		\caption{Load disturbance and univariate attack.}\label{subfig31:uasd}
	\end{subfigure}
	~
	\begin{subfigure}[t]{0.46\textwidth}
		\centering
		\includegraphics[scale=0.44]{fig4ad_a1stl_p08_2509.pdf}
		\caption{Load disturbance and univariate attack.}\label{subfig41:uasd}
	\end{subfigure}
	\\ 
	\begin{subfigure}[t]{0.46\textwidth}
		\centering
		\includegraphics[scale=0.44]{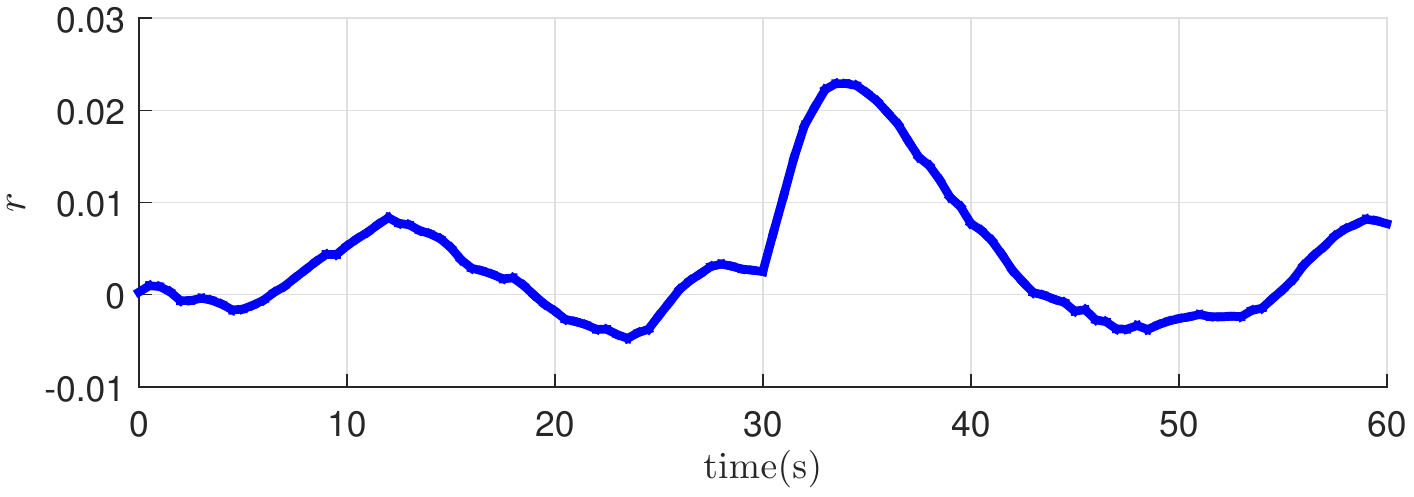}
		\caption{Residual of single instance.}\label{subfig32:ro}
	\end{subfigure}
	~
	\begin{subfigure}[t]{0.46\textwidth}
		\centering
		\includegraphics[scale=0.435]{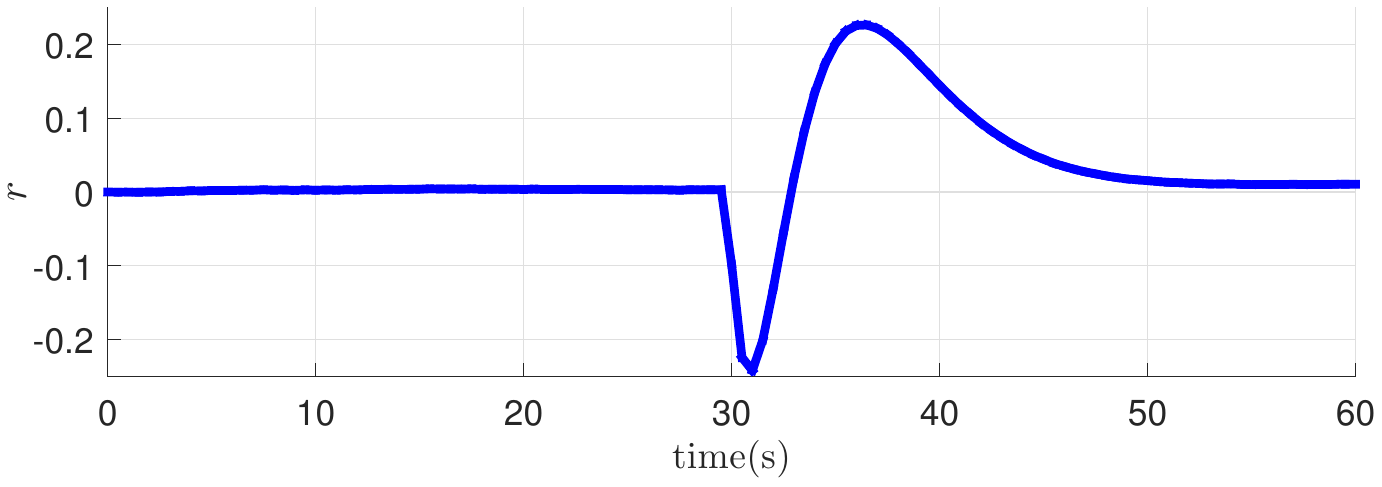}
		\caption{Residual of single instance.}\label{subfig33:rrob}
	\end{subfigure}	
	\\ 
	\begin{subfigure}[t]{0.46\textwidth}
		\centering
		\includegraphics[scale=0.44]{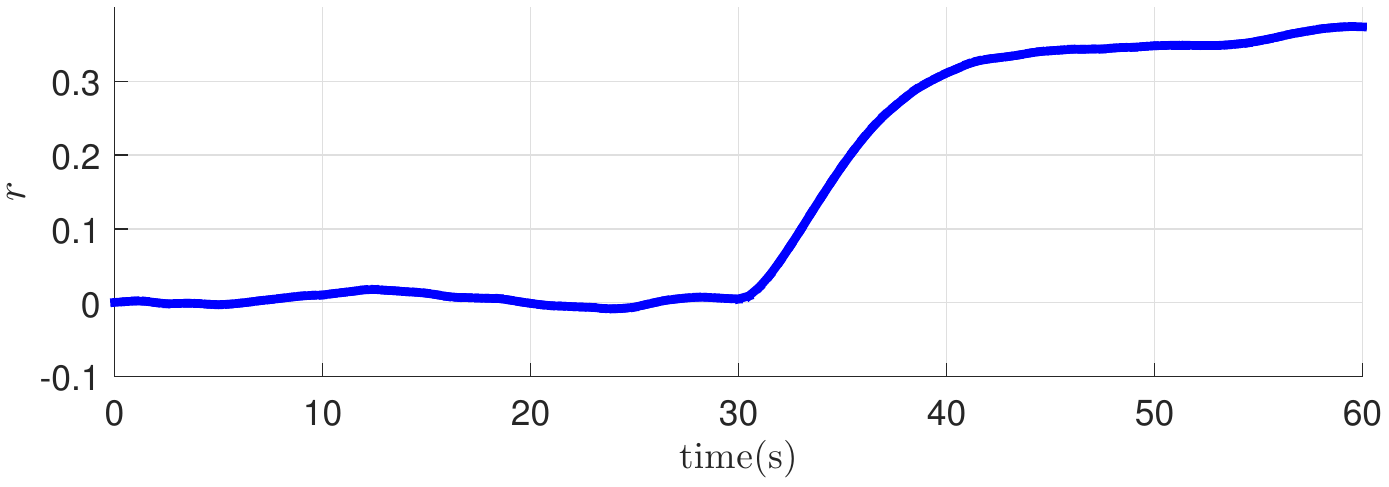}
		\caption{Residual of single instance with attack tracking capability in the steady-state value.}\label{subfig42:ross}
	\end{subfigure}
	~
	\begin{subfigure}[t]{0.46\textwidth}
		\centering
		\includegraphics[scale=0.435]{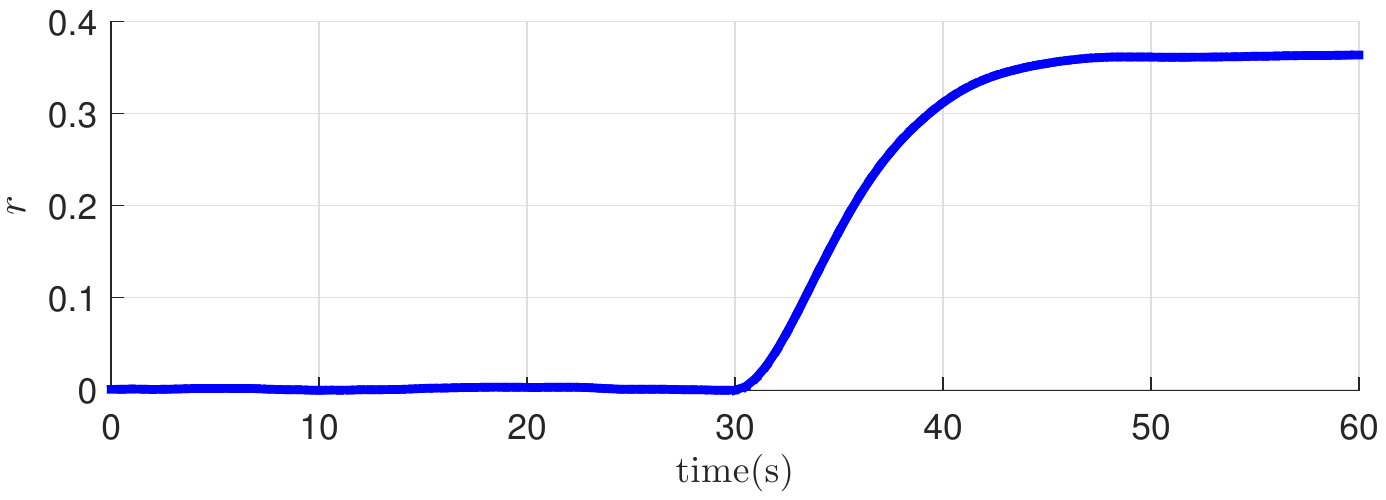}
		\caption{Residual of single instance with attack tracking capability in the steady-state value.}\label{subfig43:rrob}
	\end{subfigure}
	\\
	\begin{subfigure}[t]{0.46\textwidth}
		\centering
		\includegraphics[scale=0.44]{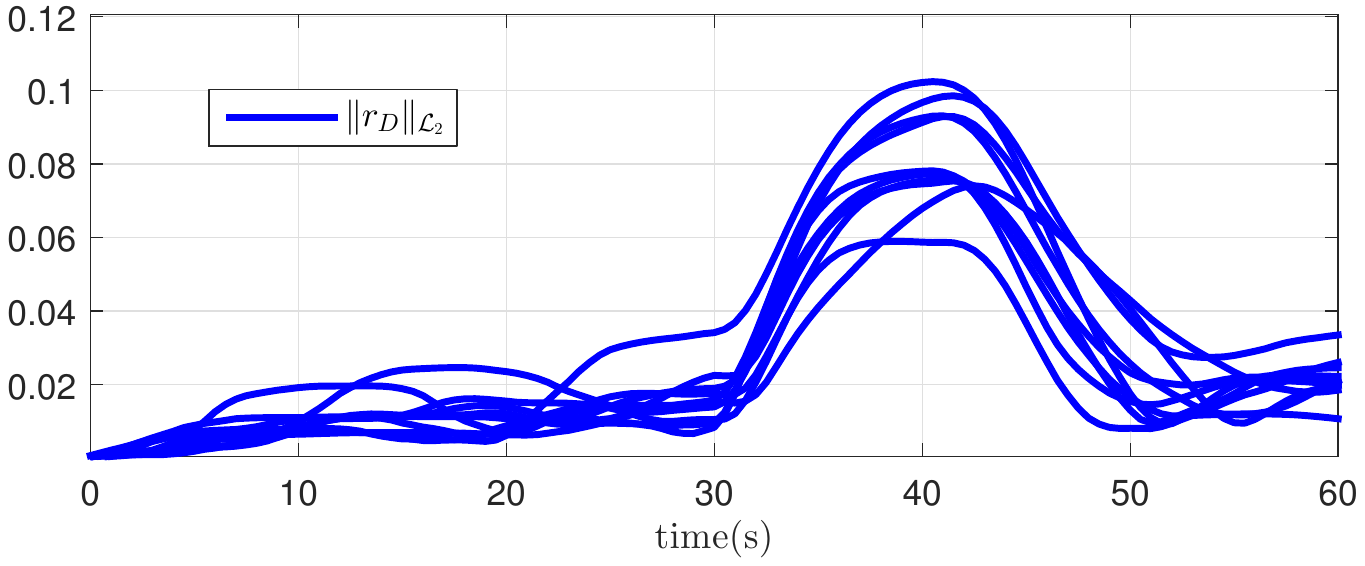}
		\caption{Energy of residual of multiple instances.}\label{subfig51:ron210}
	\end{subfigure}
	~
	\begin{subfigure}[t]{0.46\textwidth}
		\centering
		\includegraphics[scale=0.435]{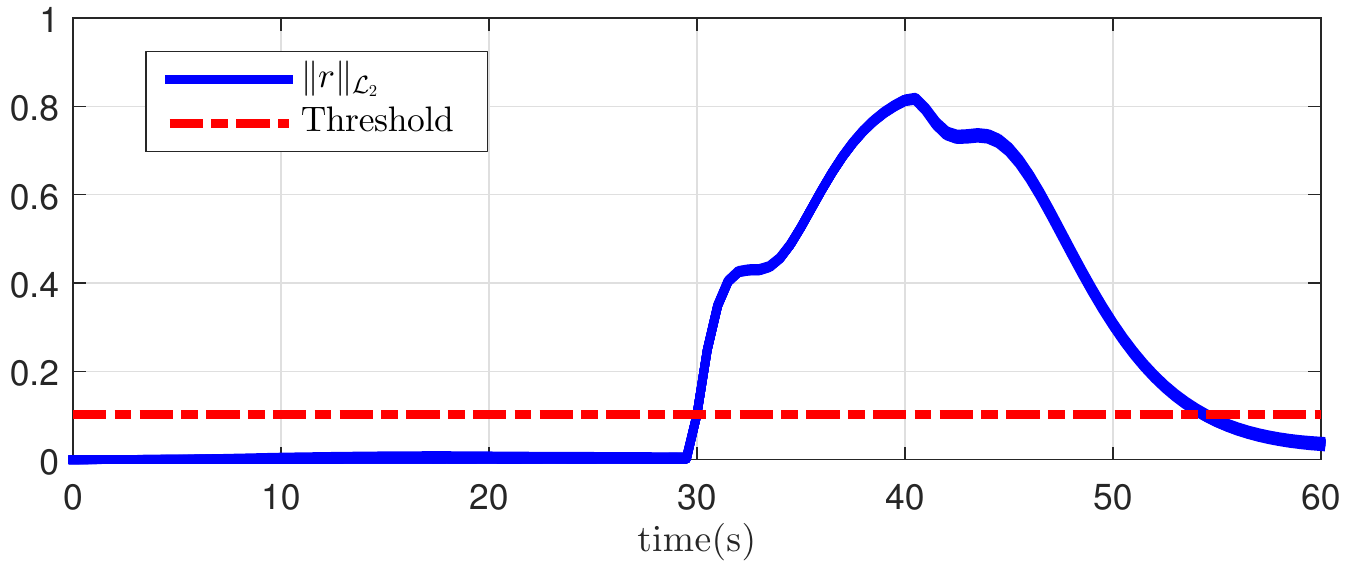}
		\caption{Energy of residual of multiple instances.}\label{subfig61:rrobn210}
	\end{subfigure}
	\caption{The left: pure model-based filter under univariate attack and model mismatch. It is derived by letting $Q_{i} = 0$ in (13), which can be transformed into finite LPs. The right: Data-assisted model-based filter under univariate attack and model mismatch. It is derived by (13) where $Q_{i}$ is from (12), and it is essentially a family of convex QPs.} %
	\label{fig:resfig_puremodel_sa}
\end{figure}
%

\section*{Appendix B: Simulation Results for the Univariate Attack Scenario} 

In Section IV-B, the simulation results of our data-assisted model-based filter and the pure model-based filter in the scenario of multivariate attacks have been illustrated in Figure 3 and Figure 4. Next in Figure~\ref{fig:resfig_puremodel_sa} 
we present the simulation results of our proposed filter and the pure model-based filter under the univariate attack. We can see that our data-assisted model-based filter has significant improvements in the regards of mitigating the effects from model mismatch on the residual; compare the figures on the right sides with the ones on the left. 
Besides, in its non-zero steady-state behavior, it can approximate the attack value (from Figure~\ref{subfig43:rrob}), while the pure model-based filter without a data-assisted perspective fails by triggering false alarms. Figure~\ref{subfig51:ron210} and~\ref{subfig61:rrobn210} provides the residual results under 10 different realizations of load disturbances in the ``testing phase''. 
Similar to Figure 3 and 4 in Section IV, Figure~\ref{subfig51:ron210} and~\ref{subfig61:rrobn210} depicts the ``energy'' of the residual signals for the last $10 \mathrm{s}$. Note that with the same rule adopted in 
Figure 4c, 
in Figure~\ref{subfig61:rrobn210} the threshold is set to $\tau^{\star} + 0.1$, where the square of $\tau^{\star}$ equals to the maximum value of $\bar{N}Q_{i}\bar{N}$ in the 100 training instances ($i \in \{1, \, \cdots, \, 100\}$, and the added value is to avoid possible false alarms. Then a univariate attack is said to be detected when the value of $\| r \|_{\mathcal{L}_{2}}[\cdot]$ is beyond this threshold. Thus as concluded in the main part of the article, our proposed diagnosis filter which can be applied to the high-fidelity simulator could generate residual ``alerts'' successfully for the occurrence of univariate attacks, and keep the effect from the model mismatch minimized. 

	\bibliographystyle{siam}
	\bibliography{literature_kp}
\end{document}